\documentclass{article}
\usepackage{lineno,hyperref}
\modulolinenumbers[5]

\usepackage{floatrow}
\usepackage{authblk}
\usepackage{graphicx}
\usepackage{subfigure} 
\usepackage{amsmath, amsfonts, amsthm, amssymb, bm}
\usepackage{color}
\usepackage{tabularx,hhline}
\usepackage{multirow}
\usepackage{booktabs}
\usepackage{fullpage}[2cm]
\usepackage{algorithm}
\usepackage{algorithmic}
\usepackage{enumerate}
\usepackage{cancel}
\usepackage{xspace}

\usepackage[margin = 1.5cm, font = small, labelfont = bf, labelsep = period]{caption}
\usepackage[font = small, labelfont = bf, labelsep = period]{caption}
\newfloatcommand{capbtabbox}{table}[][\FBwidth]

\newcommand{\subjectto}{\textnormal{subject to}}
\newcommand{\st}{\textnormal{s.t.}}

\newcommand{\maximize}{\textnormal{maximize}}
\newcommand{\minimize}{\textnormal{minimize}}
\newcommand{\norm}[1]{\left\lVert#1\right\rVert}

\DeclareMathOperator*{\diag}{diag}

\newcommand{\eg}{\textit{e.g.}}
\newcommand{\ie}{\textit{i.e.}}

\newcommand{\RR}{\mathbb{R}}
\newcommand{\ZZ}{\mathbb{Z}}

\DeclareMathAlphabet{\mathscr}{U}{dutchcal}{m}{n}
\SetMathAlphabet{\mathscr}{bold}{U}{dutchcal}{b}{n}
\DeclareMathAlphabet{\mathbscr} {U}{dutchcal}{b}{n}

\newcommand{\Ab}{{\bm{\mathcal A}}}
\newcommand{\bb}{{\bm{\mathscr b}}}
\newcommand{\Sb}{{\bm{\mathcal S}}}
\newcommand{\tb}{{\bm{\mathscr t}}}

\newcommand{\Kh}{{K'}}
\newcommand{\Jh}{{J'}}
\newcommand{\Xih}{{\Xi'}}
\newcommand{\slemma}{$\mathcal S$-lemma\xspace}
\newcommand{\X}{\mathcal{X}}

\newcommand{\xih}{{\bm \xi}'}

\newcommand{\tr}{\textup{tr}}

\newtheorem{thm}{Theorem}
\newtheorem{prop}{Proposition}
\newtheorem{lem}{Lemma}

\newtheorem{ex}{Example}
\newtheorem{rem}{Remark}

\begin{document}

\title{Robust Quadratic Programming with Mixed-Integer Uncertainty}
\author[]{Areesh Mittal}
\author[]{Can Gokalp}
\author[]{Grani A.~Hanasusanto}

\affil[]{\small Graduate Program in Operations Research and Industrial Engineering, The University of Texas at Austin, USA}

\maketitle
%

\begin{abstract}
	 We study robust convex quadratic programs where the uncertain problem parameters can contain both continuous and integer components. Under the natural boundedness assumption on the uncertainty set, we show that the generic problems are amenable to exact copositive programming reformulations of polynomial size.  These convex optimization problems are NP-hard but admit a conservative semidefinite programming (SDP) approximation that can be solved efficiently.  We prove that the popular approximate \slemma method---which is valid only in the case of continuous uncertainty---is weaker than our approximation.  We also show that all results can be extended to the two-stage robust quadratic optimization setting if the problem has complete recourse.  We assess the effectiveness of our proposed SDP reformulations and demonstrate their superiority over the state-of-the-art solution schemes on instances of least squares, project management, and multi-item newsvendor problems. 
\end{abstract}

\section{Introduction}\label{sec:introduction}
A wide variety of decision making problems in engineering, physical, or economic systems can be formulated as convex quadratic programs of the form 
\begin{equation}
\label{eq:QCQP}
\begin{array}{clll}
\displaystyle\minimize&\displaystyle  \norm{\bm A (\bm x) \bm \xi}^2 +  \bm b(\bm x)^\top\bm \xi + c (\bm x)\\
\subjectto& \displaystyle {\bm x\in\mathcal X}.
\end{array}
\end{equation}
Here,  $\mathcal X\subseteq\RR^D$ is the feasible set of the decision vector $\bm x$
and is assumed to be described by a polytope,  $\bm{\xi}\in\RR^K$ is a vector of exogenous problem parameters, $\bm A(\bm x):\mathcal X\rightarrow\RR^{M\times K}$ and $\bm b(\bm x):\mathcal X\rightarrow\RR^K$ are matrix- and vector-valued affine functions, respectively, while $c(\bm x):\mathcal X\rightarrow\RR$ is a convex quadratic function. The objective of problem \eqref{eq:QCQP} is to determine the best decision $\bm x\in\mathcal X$ that minimizes the quadratic function $\norm{\bm A (\bm x) \bm \xi}^2 +  \bm b(\bm x)^\top\bm \xi + c (\bm x)$. The generic formulation~\eqref{eq:QCQP} includes the class of  linear programming problems~\cite{schrijver1998theory} as a special case (when $\bm A=\bm 0$), and has numerous important applications, e.g., in portfolio optimization~\cite{markowitz1952portfolio},  least squares regression~\cite{golub2012matrix}, supervised classification~\cite{cortes1995support}, optimal control~\cite{rockafellar1990generalized}, etc. In addition to their exceptional modeling power, quadratic optimization problems of the form~\eqref{eq:QCQP} are attractive as they can be solved efficiently using standard off-the-shelf solvers. 

In many situations of practical interest, the exact values of the parameters $\bm{\xi}$ are unknown when the decisions are made and can only be estimated through limited historical data. Thus, they are subject to potentially significant errors that can adversely impact the out-of-sample performance of an optimal solution~$\bm x$. One popular approach to address decision problems under uncertainty is via \emph{robust optimization}~\cite{ben2009robust}. In this setting, we assume that the vector of uncertain parameters $\bm{\xi}$  lies within a prescribed uncertainty set $\Xi$ and we replace the objective function of \eqref{eq:QCQP} with the \emph{worst-case} function given by
\begin{equation}
\label{eq:wc_functions}
\sup_{\bm{\xi}\in\Xi}\norm{\bm A (\bm x) \bm \xi}^2 + \bm b (\bm x)^\top\bm \xi + c (\bm x). 
\end{equation}
This optimization problem yields a solution $\bm x \in \mathcal X$ that minimizes the quadratic objective function under the most adverse uncertain parameter realization $\bm{\xi}\in\Xi$.  


Robust optimization models are appealing as they require minimal assumptions on the description of the uncertain parameters and because they often lead to efficient solution schemes. In a linear programming setting, the resulting robust optimization problems are tractable for many relevant uncertainty sets and have been broadly applied to problems in engineering, finance, machine learning, and operations management~\cite{ben1998robust, bertsimas2011theory, gorissen2015practical}. Tractable reformulations for robust quadratic programming problems are derived in \cite{goldfarb2003robust,Ahmadreza2017} for the particular case when the quadratic functions (in $\bm x$) exhibit a concave  dependency in the uncertain parameters $\bm\xi$. When the functions are convex in both  $\bm x$ and $\bm{\xi}$ as we consider in this paper, the corresponding robust problems are generically NP-hard if the uncertainty set  is defined by a polytope, but become tractable---by virtue of the \emph{exact} \slemma---if the uncertainty set is defined by an ellipsoid~\cite{ben1998robust, el1997robust}.  Tractable approximation schemes have also been proposed for the  standard setting that we consider in this paper.  If the uncertainty set is described by a finite intersection of ellipsoids then a conservative semidefinite programming (SDP) reformulation is obtained by leveraging the \emph{approximate} \slemma~\cite{ben2002robust}. In  \cite{bertsimas2006tractable}, a special class of functions is introduced to approximate the quadratic terms in \eqref{eq:wc_functions}. The arising robust optimization problems are tractable if the uncertainty sets are defined through affinely transformed norm balls. 
In \cite{Ahmadreza2017}, conservative and progressive SDP approximations are devised  by replacing each quadratic term in \eqref{eq:wc_functions} with linear upper and lower bounds, respectively.  

Most of the existing literature in robust optimization assume that the uncertain problem parameters are continuous and reside in a tractable conic representable set $\Xi$. However, certain applications require the use of mixed-integer uncertainty. Such decision  problems  arise prominently in the supply chain context where demands of  non-perishable products are more naturally represented as integer quantities and in the discrete choice modeling context where the outcomes are chosen from a discrete set of alternatives. 
 Other pertinent examples include  robust optimization applications in logistic regression~\cite{shafieezadeh2015distributionally}, classification problems with noisy labels~\cite{caramanis201114,xu2009robustness} and network optimization~\cite{ahipasaoglu2016distributionally, wiesemann2012robust}. If the uncertain parameters contain mixed-integer components then the problem becomes  computationally formidable even in the simplest setting. Specifically, if all functions are affine in~$\bm{\xi}$ and the uncertain problem parameters are described by binary vectors, then computing the worst-case values in \eqref{eq:wc_functions} is already NP-hard~\cite{GJ79:ComputersIntractability}. The corresponding robust version of~\eqref{eq:QCQP} is tractable only in a few contrived situations, \eg, when the uncertainty set possesses a {total unimodularity} property or is described by the convex hull of polynomially many integer vectors~\cite{ben1998robust}. Perhaps due to these limitations, there are currently very few results in the literature that provide a  systematic and rigorous way to handle generic robust optimization problems with mixed-integer uncertainty. In this paper, we first reformulate the original problem as an equivalent finite-dimensional conic program of polynomial size, which absorbs all the difficulty in its cone, and then replace the cone with tractable inner approximations. An alternate way to handle integer uncertain parameters can be to solve the problem by simply ignoring the integrality assumption. However, doing so adds undesired conservativeness to the uncertainty set. Indeed, in our numerical experiments, we demonstrate that ignoring the integrality assumption on the uncertain parameters leads to overly conservative solutions.

Optimization problems under uncertainty may also involve adaptive recourse decisions which are taken once the uncertain parameters are realized \cite{ben2009robust, shapiro2014lectures}. This setting gives rise to difficult min-max-min optimization problems which are generically NP-hard even if both the first- and the second-stage cost functions are affine in $\bm x$ and $\bm{\xi}$ \cite{BGGN04:LDR}. Thus, they can only be solved approximately, either by employing discretization schemes which approximate the continuum of the uncertainty space with finitely many points~\cite{hadjiyiannis2011scenario, kleywegt2002sample, shapiro2003monte} or by employing decision rule methods, which restrict the set of all possible recourse decisions to simpler parametric forms in $\bm{\xi}$~\cite{BGGN04:LDR, georghiou2015generalized, goh2010distributionally}. We refer the reader to \cite{delage2015robust} for a comprehensive review of recent results in adaptive robust optimization. In this paper, we consider two-stage robust optimization problems with quadratic first- and second-stage objective function and a mixed-integer uncertainty set. We show that if the problem has \emph{complete recourse}, then it can be reformulated as a conic program---which is amenable to tractable approximations.

The conic programming route that we take here to model optimization problems under uncertainty has previously been traversed. In \cite{NRZ11:mixed01}, completely positive programming reformulations are derived to compute best-case expectations of mixed zero-one linear programs under first- and second-order moment information on the joint distributions of the uncertain parameters. 
This result has been  extended and applied to other pertinent settings such as in stochastic appointment scheduling problems, discrete choice models, random walks and sequencing problems \cite{kong2013scheduling,li2014distributionally,natarajan2017reduced}. Recently, equivalent copositive programming reformulations are derived for generic two-stage robust linear programs \cite{hanasusanto2018conic,xu2018copositive}. The resulting optimization problems are amenable to conservative semidefinite programming reformulations which are often stronger than the ones obtained from employing quadratic decision rules on the recourse function. In \cite{gao2016disruption}, the authors provide completely positive reformulation for a two-stage distributionally robust supply chain risk mitigation problem. They allow some components of $\bm \xi$ to be binary, but assume precise knowledge of the first- and the second-order moments of the distribution of $\bm \xi$. The objective function that they consider is quadratic in the second-stage decision variables but affine in $\bm \xi$. In contrast to \cite{gao2016disruption},  we assume no information about the distribution of $\bm \xi$, other than the support. Furthermore, we allow the objective function to be quadratic in the decision variables, as well as in $\bm \xi$, which helps us model a more general class of robust problems, \eg, robust least squares \cite{el1997robust}.

In this paper, we advance the state-of-the-art in robust optimization along several
directions. We summarize our main contributions as follows: 
\begin{enumerate}
\item We prove that any robust convex quadratic program can be reformulated as a copositive program of polynomial size if the uncertainty set is given by a \emph{bounded} mixed-integer polytope. We  further show that the exactness result can be extended to the two-stage robust quadratic optimization setting  if the problem has \emph{complete recourse}. 
\item By employing the hierarchies of semidefinite representable cones to approximate the copositive cones, we obtain  sequences of tractable conservative approximations for the robust problem. These approximations can be made to have any arbitrary accuracy. We prove that even the simplest of these approximations is stronger than the well-known approximate \slemma method if the problem instance has only continuous uncertain parameters. Furthermore, when some uncertain parameters are restricted to take integer values, the approximate \slemma method is not applicable, while our method still generates a high-quality conservative solution.
\item We compare our approximation method to other state-of-the-art approximation schemes through extensive numerical experiments. We show that our approximation method generates better estimates of worst-case cost and yields less conservative solutions. We also demonstrate that ignoring the integrality assumption on the uncertainty set may lead to inferior solutions to the robust problem.
\item To the best of our knowledge, we are the first to provide an exact conic programming reformulation and to propose tractable semidefinite programming approximations for well-established classes of one-stage and two-stage robust quadratic programs. 
\end{enumerate}



The remainder of the paper is structured as follows. We formulate and discuss the generic robust quadratic programs in Section \ref{sec:formulation}. We then derive the copositive programming reformulation in Section \ref{sec:conic reformulation}. Section \ref{sec:SDP_formulation} develops a conservative SDP reformulation and provides a theoretical comparison with the popular approximate \slemma method.  In Section \ref{sec:extensions}, we extend the results of Section \ref{sec:conic reformulation} along several directions including two-stage robust quadratic optimization. We demonstrate the impact of our proposed reformulation via numerical experiments in Section \ref{sec:experiments}, and finally, we conclude in Section \ref{sec:conclusion}. 

\paragraph{Notation:} We use $\ZZ\ (\ZZ_+)$ to denote the set of (non-negative) integers. For any positive integer $I$, we use $[I]$ to denote the index set $\{1,\dots,I\}$. We use $\norm{.}_p$ to denote the $l_p$-norm. We drop the subscript and write $\norm{.}$ when referring to the $l_2$-norm. The identity matrix and the vector of all ones are denoted by $\mathbb I$ and $\mathbf e$, respectively. The dimension of such matrices will be clear from the context. We denote by $\tr(\bm M)$ the trace of a square matrix  $\bm M$. 
For a vector $\bm v$, $\diag(\bm v)$ denotes the diagonal matrix with $\bm v$ on its diagonal; whereas for a square matrix $\bm M$, $\diag(\bm M)$ denotes the vector comprising the diagonal elements of~$\bm M$. We define $\bm P \circ \bm Q$ as the Hadamard product (element-wise product) of two matrices $\bm P$ and $\bm Q$ of the same size. For any  integer $Q\in \ZZ_+$, we define   
$\mathbf v_Q=[2^0\;2^1\;\cdots\;2^{Q-1}]^\top$ as the vector comprising all $q$-th powers of 2, for $q= 0,1,\ldots,Q-1$. 
We define  by 
$\mathbb S^K$ ($\mathbb S_+^K$) the space of all symmetric (positive semidefinite) matrices in $\RR^{K\times K}$. 
The cone of copositive matrices is denoted by $\mathcal C=\{\bm M\in\mathbb S^K:\bm{\xi}^\top\bm{M}\bm{\xi}\geq 0\;\forall\bm{\xi}\geq \bm 0\}$, while its dual cone, the cone of completely positive matrices, is denoted by $\mathcal C^*=\{\bm{M}\in\mathbb S^K:\bm M=\bm B\bm B^\top \text{ for some } \bm B\in\RR_+^{K\times g(K)}\}$, where $g(K)= \max\{{K+1\choose 2}-4,K\}$ \cite{SBBJ15:CP-RANK}.
For any $\bm P, \bm Q\in\mathbb S^K$, the relations $\bm P\succeq\bm Q$, $\bm P\succeq_{\mathcal C}\bm Q$, and $\bm P\succeq_{\mathcal {C}^*}\bm Q$ indicate that $\bm P-\bm Q$ is an element of $\mathbb S_+^K$, $\mathcal C$, and $\mathcal C^*$, respectively.

\section{Problem Formulation}\label{sec:formulation}
We study robust convex quadratic programs (RQPs) of the form 
\begin{equation}
\label{eq:RO}
\begin{array}{clll}
\displaystyle\minimize&\displaystyle\sup_{\bm{\xi}\in\Xi}  \norm{\bm A (\bm x) \bm \xi}^2 +  \bm b(\bm x)^\top\bm \xi + c (\bm x)\\
\subjectto& \displaystyle {\bm x\in\mathcal X},
\end{array}
\end{equation}
where the set $\mathcal X$ and the functions  $\bm A(\bm x):\mathcal X\rightarrow\RR^{M\times K}$, $\bm b(\bm x):\mathcal X\rightarrow\RR^K$, and $c(\bm x):\mathcal X\rightarrow\RR$ have the same definitions as those in \eqref{eq:QCQP}. 
The vector $\bm{\xi}\in\RR^K$ comprises all the uncertain problem parameters and is assumed to belong to the uncertainty set $\Xi$ 
given by a bounded mixed-integer polyhedral set 
\begin{equation}
\label{eq:uncertainty_set}
\Xi=\left\{\bm \xi\in\RR_+^K:
\begin{array}{l}\bm S\bm \xi= \bm t\\
\xi_{\ell}\in\ZZ\quad\forall \ell\in[L]
\end{array}\right\},
\end{equation}
where $\bm S\in\RR^{J\times K}$ and $\bm t\in\RR^{J}$. We assume without loss of generality that the first $L$ elements of $\bm\xi$ are integer, while the remaining $K-L$ are continuous. Since $\Xi$ is bounded, we may further assume that there exists a scalar integer $Q\in\ZZ_+$ such that $\xi_l\in \{0,\cdots,2^{Q}-1\}$  for every $\ell\in[L]$. Note that the quantity $Q$ is bounded by a polynomial function in the bit length of the description of $\bm S$ and $\bm t$. 
\begin{ex}[Robust Portfolio Optimization] Consider the classical Markowitz mean-variance portfolio optimization problem
\begin{equation}
\label{eq:Markowitz}
\begin{array}{clll}
\minimize & \displaystyle\bm x^\top{\bm{\Sigma}}\bm x-\lambda{\bm{\mu}}^\top\bm x\\\
\subjectto & \displaystyle\bm x\in\Delta^K,
\end{array}
\end{equation}
where $\Delta^K$ is the unit simplex in $\RR^K$, $\lambda\in[0,\infty)$ is the prescribed risk tolerance level of the investor, while $\bm\mu\in\RR^K$ and $\bm\Sigma\in\mathbb S^K$ are the true mean and covariance matrix of the asset returns, respectively. The objective of this problem is to determine the best vector of weights $\bm x\in\Delta^K$ that maximizes the  mean portfolio return ${\bm{\mu}}^\top\bm x$ and that  also minimizes the portfolio risk that is captured by the variance term $\bm x^\top{\bm{\Sigma}}\bm x$. Here, the trade-off between these two terms is controlled by the scalar $\lambda$ in the objective function. 

In practice, the true values of the parameters $\bm\mu$ and $\bm{\Sigma}$ are unknown and can only be estimated by using the available $N$ historical asset returns $\{\hat{\bm{\xi}}_n\}_{n\in[N]}$,  as follows:
\begin{equation*}
\hat{\bm{\mu}}=\frac{1}{N}\sum_{n\in[N]}{\hat{\bm\xi}_n}\quad\text{and}\quad \hat{\bm{\Sigma}}=\frac{1}{N-1}\sum_{n\in[N]}\left({\hat{\bm\xi}_n}-\hat{\bm{\mu}}\right)\left({\hat{\bm\xi}_n}-\hat{\bm{\mu}}\right)^\top.
\end{equation*}
In the robust optimization setting, we assume that the precise location of  each sample point $\hat{\bm{\xi}}_n$ is uncertain and is only known to belong to a prescribed uncertainty set $\Xi_n$ containing $\hat{\bm{\xi}}_n$. 
To bring the resulting  problem into the standard form \eqref{eq:RO}, we introduce the expanded uncertainty set
\begin{equation*}
\Xi=\left\{\left((\hat{\bm \xi}_n)_{n\in[N]},(\hat{\bm \chi}_n)_{n\in[N]}\right)\in\RR_+^{NK+NK}:
\hat{\bm \xi}_n\in\Xi_n,\;\;\hat{\bm \chi}_n={\hat{\bm\xi}_n}-\frac{1}{N}\sum_{n'\in[N]}{\hat{\bm\xi}_{n'}}\quad\forall n\in[N]
\right\}
\end{equation*}
comprising the terms $\hat{\bm\xi}_n$ and ${\hat{\bm\xi}_n}-\hat{\bm{\mu}}$, $n\in[N]$. 
Using this uncertainty set,  we arrive at the following robust version of \eqref{eq:Markowitz}:
\begin{equation*}
\label{eq:Markowitz_robust}
	\begin{array}{clll}
	\minimize &\displaystyle\sup_{\left((\hat{\bm \xi}_n)_n,(\hat{\bm \chi}_n)_n\right)\in\Xi}\left( \frac{1}{N-1}\sum_{n\in[N]}(\hat{\bm \chi}_n^\top\bm x)^2-\frac{\lambda}{N}\sum_{n\in[N]}\hat{\bm \xi}_n^\top\bm x\right)\\
	\subjectto & \bm x\in\Delta^K.
	\end{array}
\end{equation*}	 
This problem constitutes an instance of \eqref{eq:RO} with the input parameters 
\begin{equation*}
 \bm A(\bm x)=\frac{1}{\sqrt{N-1}}\begin{bmatrix}\bm 0^\top&\cdots  &\bm 0^\top&\bm 0^\top&\cdots&\ \bm 0^\top\\
\vdots&\ddots  &\vdots&\vdots&\ddots&\vdots\\
 \bm 0^\top&\cdots  &\bm 0^\top&\bm 0^\top&\cdots&\ \bm 0^\top\\
 \bm 0^\top&\cdots  &\bm 0^\top&\bm x^\top&\cdots&\ \bm 0^\top\\
 \vdots&\ddots  &\vdots&\vdots&\ddots&\vdots\\
 \bm 0^\top&\cdots  &\bm 0^\top&\bm 0^\top&\cdots&\ \bm x^\top\\
  \end{bmatrix},\quad  \bm b(\bm x)=-\frac{\lambda}{N}\begin{bmatrix}\bm x\\\vdots\\\bm x\\ \bm 0\\\vdots\\\bm 0\end{bmatrix},\quad\text{and}\quad c(\bm x)=0. 
\end{equation*}

\end{ex}
\begin{ex}[Robust Project Crashing]
	\label{exa:project_crashing}
	Consider a project that is described by an activity-on-arc network~$\mathcal N(\mathcal V,\mathcal A)$, where $\mathcal V$ is the set of nodes  representing the events, while $\mathcal A$ is the set of arcs  representing the activities. We assume that that node with index $1$ represents the start of the project and the node with index $|\mathcal V|$ represents the end of the project. We define $d_{ij}\in[0,1]$ to be the nominal duration of the  activity $(i,j)\in\mathcal A$. Here, we assume that the durations $d_{ij}$, $(i,j)\in\mathcal A$, are already normalized so that they  take values in the unit interval.  
	
	The goal of project crashing is to determine the best resource assignments $x_{ij}$, $(i,j)\in\mathcal A$, on the activities  that minimize the project completion time or makespan. If the activity duration $d_{ij}-x_{ij}$ represents the length of the arc $(i,j)$, then the project completion time can be determined by computing the length of the longest path from the start node to the end node. We can formulate project crashing as the optimization problem 
	\begin{equation*}
	\label{eq:robust_crashing}
	\begin{array}{clll}
	\minimize & \displaystyle\sup_{\bm z\in\mathcal Z}\;\sum_{(i,j)\in\mathcal A}({d}_{ij}-x_{ij})z_{ij} \\
	\subjectto & \displaystyle\bm x\in\mathcal X,
	\end{array}
	\end{equation*}
	where
	\begin{equation*}
	\mathcal Z=\left\{\bm z\in\{0,1\}^{|\mathcal A|}:\sum_{j:(i,j)\in\mathcal A}z_{ij}-\sum_{j:(j,i)\in\mathcal A}z_{ji}=\left\{\begin{array}{ll}
	1&\text{if }i=1\\
	-1&\text{if }i=|\mathcal V|\\
	0&\text{if otherwise}
	\end{array}\right.,\quad\forall i\in \mathcal V\right\}. 
	\end{equation*}
If the task durations $\bm{d}$ are uncertain and are only known to belong to the prescribed uncertainty set $\mathcal D\subseteq[0,1]^{|\mathcal A|}$, then we arrive at the robust optimization problem 
\begin{equation}
\label{eq:robust_project_crashing}
\begin{array}{clll}
\minimize & \displaystyle\sup_{\bm d\in\mathcal D}\left(\sup_{\bm z\in\mathcal Z}\;\sum_{(i,j)\in\mathcal A}({d}_{ij}-x_{ij})z_{ij}\right) \\
\subjectto & \displaystyle\bm x\in\mathcal X.
\end{array}
\end{equation}
By combining the suprema over $\mathcal D$ and $\mathcal Z$, and linearizing the bilinear terms $d_{ij}z_{ij}$, $(i,j)\in\mathcal A$, we can reformulate the objective of this problem as
\begin{equation}
\label{eq:objective_project_crashing}
\begin{array}{rlll}
\displaystyle\sup_{\bm d\in\mathcal D}\sup_{\bm z\in\mathcal Z}\;\sum_{(i,j)\in\mathcal A}({d}_{ij}-x_{ij})z_{ij}=&\displaystyle\sup_{(\bm d,\bm z,\bm q)\in\Xi}\;\mathbf e^\top\bm q-\bm x^\top\bm z,
\end{array}
\end{equation}
where
\begin{equation}
\label{eq:uncertainty_set_project_crashing}
\Xi=\left\{(\bm d,\bm z,\bm q)\in\mathcal D\times\mathcal Z\times\RR_+^{|\mathcal A|}:\bm q\leq\bm z,\;\bm q\leq\bm d,\;\bm q\geq \bm d-\mathbf e+\bm z\right\}. 
\end{equation}
Using the new objective function \eqref{eq:objective_project_crashing} and uncertainty set \eqref{eq:uncertainty_set_project_crashing}, the resulting robust optimization problem constitutes an instance of \eqref{eq:RO} with the input parameters 
$\bm A(\bm x)=\bm 0$,  $\bm b(\bm x)=[
\bm 0^\top\;\; -\bm x^\top \;\; \mathbf e^\top
]^\top$, and $c(\bm x)=0$.

\end{ex}

In the remainder of the paper, for any fixed $\bm x\in\mathcal X$, we define the mixed-integer quadratic program
\begin{equation}
\label{eq:quadratic_program}
Z(\bm x)=\sup_{\bm{\xi}\in\Xi}\norm{\bm A (\bm x) \bm \xi}^2 + \bm b (\bm x)^\top\bm \xi + c (\bm x),
\end{equation}
which corresponds to the inner subproblem in the objective of \eqref{eq:RO}. 
We may therefore represent \eqref{eq:RO} as
\begin{equation*}
\begin{array}{clll}
\displaystyle\minimize&\displaystyle Z(\bm x)\\
\subjectto& \displaystyle {\bm x\in\mathcal X}.
\end{array}
\end{equation*}
In the next section, we derive exact copositive programming reformulation for evaluating $Z(\bm x)$. By substituting $Z(\bm x)$ with the emerging copositive program, we obtain an equivalent finite-dimensional convex reformulation for the RQP~\eqref{eq:RO} that is principally amenable to numerical solution. 

\section{Copositive Programming Reformulation}\label{sec:conic reformulation}
In this section, we derive an equivalent copositive programming reformulation for \eqref{eq:RO} by adopting the following steps. For any fixed $\bm x\in \mathcal X$, we first derive a copositive upper bound on $Z(\bm x)$. We then show that the resulting reformulation is in fact exact under the boundedness assumption on the uncertainty set $\Xi$.

\subsection{A Copositive Upper Bound on  \texorpdfstring{$Z(\bm x)$}{}}
To derive the  copositive reformulation, we leverage the following result by Burer \cite{Burer09:copositive} which enables us to reduce a generic mixed-binary quadratic program into an equivalent conic program of polynomial size. 
\begin{thm}[{\cite[Theorem 2.6]{Burer09:copositive}}]
	\label{thm:burer}
	The mixed-binary quadratic program
	\begin{equation}
	\label{eq:standard_QP}
	\begin{array}{clll}
	\displaystyle\maximize&\displaystyle\bm{\xi}^\top\bm{Q}\bm{\xi}+\bm r^\top\bm{\xi} \\
	\subjectto& \displaystyle \bm{\xi}\in\RR_+^P\\
	&\displaystyle \bm F\bm{\xi}=\bm g\\
	&\displaystyle \xi_\ell\in\{0,1\}&\forall \ell\in\mathcal L
	\end{array}
	\end{equation}
	is equivalent to the completely positive program
	\begin{equation*}
	\begin{array}{clll}
	\displaystyle\maximize&\displaystyle\tr(\bm{\Omega}\bm Q)+\bm r^\top\bm{\xi} \\
	\subjectto& \displaystyle \bm{\xi}\in\RR_+^P,\;\bm{\Omega} \in\mathbb S_+^P\\
	&\displaystyle \bm F\bm{\xi}=\bm g,\;\;\diag(\bm F\bm\Omega\bm F^\top)=\bm g\circ\bm g\\ 
	&\displaystyle \xi_\ell=\Omega_{\ell\ell}\qquad\forall \ell\in\mathcal L\\
	&\begin{bmatrix}  \bm\Omega &  \bm\xi\\
	\bm\xi^\top & 1\\
	\end{bmatrix}\succeq_{\mathcal C^*} \bm 0,
	\end{array}
	\end{equation*}
	where $\mathcal L\subseteq[P]$, and it is implicitly assumed that $\xi_{\ell}\leq 1,\;\ell\in\mathcal L$, for any $\bm{\xi}\in\RR_+^P$ satisfying $\bm F\bm{\xi}=\bm g$. 
\end{thm}
We also rely on the following standard result which 
allows us to represent a scalar integer variable using only logarithmically many binary variables \cite{watters1967letter}.
\begin{lem}
	\label{lem:integer_binary}
If $\xi$ is a scalar integer decision variable taking values in $\{0,\cdots,2^{Q}-1\}$, with $Q\in\ZZ_+$, then we can reformulate it concisely by employing $Q$ binary decision variables $\chi_1,\cdots,\chi_{Q}\in\{0,1\}$, as follows:
\begin{equation*}
\xi=\sum_{q\in[Q]}2^{q-1}\chi_{q}=\mathbf v_Q^\top\bm{\chi}.
\end{equation*}
\end{lem}
Using Theorem \ref{thm:burer} and Lemma \ref{lem:integer_binary}, we are now ready to state our first result. 

\begin{prop} 
	\label{prop:CPP}
	For any fixed decision $\bm x\in\X$ the optimal value of the quadratic maximization problem \eqref{eq:quadratic_program} coincides with the optimal value of the completely positive program
	\begin{equation}
	\label{eq:CPP}
	\begin{array}{clll}
	\displaystyle  Z(\bm x) = &\displaystyle \sup&\tr\left(\Ab(\bm x) \bm \Omega \Ab(\bm x) ^\top\right) + \bb(\bm x)^\top\xih + c (\bm x)  \\
	&\st&\displaystyle {\xih}\in\RR_+^{{\Kh}},\;\bm{\Omega} \in\mathbb S_+^{{\Kh}}\\
	&& \Sb  \xih = \tb,\;\diag(\Sb\bm \Omega\Sb^\top)=\tb\circ\tb\\
	&& {\xi}_{\ell}'=\Omega_{\ell \ell} \quad \forall \ell\in\left[L Q\right]\\
	&&\begin{bmatrix}  \bm\Omega &  \xih\\
	 \xih^\top & 1\\
	\end{bmatrix}\succeq_{\mathcal C^*} \bm 0,
	\end{array}
	\end{equation}
	where 
	\begin{equation}
	\label{eq:expanded_parameters}
	\begin{array}{c}
	\Sb=\begin{bmatrix}
	\bm 0 & \cdots & \bm 0 &  \bm 0 & \cdots & \bm 0 & \bm S \\
	-\mathbf v_{Q}^\top & \cdots & \bm 0^\top & \bm 0^\top & \cdots & \bm 0^\top & \mathbf e_1^\top \\
	\vdots & \ddots & \vdots & \vdots & \ddots & \vdots & \vdots \\
	\bm 0^\top & \cdots &  -\mathbf v_{Q}^\top & \bm 0^\top & \cdots & \bm 0^\top & \mathbf e_L^\top \\
	\mathbb I & \cdots & \bm 0 &\mathbb I & \cdots & \bm 0& \bm 0\\
	\vdots & \ddots & \vdots & \vdots & \ddots & \vdots & \vdots \\
	\bm 0 & \cdots & \mathbb I &\bm 0 & \cdots & \mathbb I & \bm 0  \\
	\end{bmatrix}\in\RR^{\Jh\times \Kh},\quad 
	\tb=\begin{bmatrix}
	\bm t \\
	0\\
	\vdots \\
	0\\
	\mathbf e\\
	\vdots \\
	\mathbf e\\
	\end{bmatrix}\in\RR^{\Jh},\\
	\quad\quad\Ab(\bm x)=\begin{bmatrix}
	\bm 0&\cdots & \bm 0 & \bm 0&\cdots & \bm 0 & \bm A(\bm x)
	\end{bmatrix}\in\RR^{M\times \Kh}\quad \text{ and }\\\quad\bb(\bm x)=\begin{bmatrix}
	\bm 0^\top&\cdots & \bm 0^\top & \bm 0^\top&\cdots & \bm 0^\top & \bm b(\bm x)^\top
	\end{bmatrix}^\top\in\RR^{\Kh},
	\end{array}
	\end{equation}
with
\begin{equation*}
		\Jh=LQ+J+L\quad\text{and}\quad\Kh=2LQ+K.
\end{equation*}
\end{prop}
\begin{proof}
Lemma \ref{lem:integer_binary} enables us to reformulate the mixed-integer quadratic program \eqref{eq:quadratic_program} equivalently as the mixed-binary quadratic program
	\begin{equation}
	\label{eq:MIBP}
	\begin{array}{clll}
	Z(\bm x)=&\displaystyle \sup&\norm{\bm A (\bm x) \bm \xi}^2 + \bm b (\bm x)^\top\bm \xi + c (\bm x) \\
	&\st&\displaystyle \bm{\xi}\in\RR_+^K,\;\bm\chi_\ell\in\{0,1\}^{Q}&\forall \ell\in[L]\\
	&&	\bm S\bm{\xi}= \bm t\\
	&&	\displaystyle\xi_\ell=\mathbf v_{Q}^\top \bm{\chi}_\ell&\forall \ell\in[L].
	\end{array}
	\end{equation}
	We now employ Theorem \ref{thm:burer} to derive the equivalent completely positive program for \eqref{eq:MIBP}. To this end, we first bring the above quadratic program into the standard form \eqref{eq:standard_QP}. We introduce the redundant linear constraints $\bm{\chi}_\ell\leq\mathbf e$, $\ell\in[L]$, 
	which are pertinent for the exactness of the reformulation, and we define new auxiliary slack variables $\bm{\eta}_\ell$, $\ell\in[L]$,  to transform these inequalities into the equality constraints $\bm\chi_\ell+\bm{\eta}_\ell=\mathbf e$, $\forall \ell\in[L]$. This yields the equivalent problem
	\begin{equation}
	\label{eq:quad_max_}
	\begin{array}{clll}
		Z(\bm x)=&\displaystyle \sup&\norm{\bm A (\bm x) \bm \xi}^2 + \bm b (\bm x)^\top\bm \xi + c (\bm x) \\
	&\st&\displaystyle \bm{\xi}\in\RR_+^K,\;\bm{\eta}_\ell\in\RR_+^Q,\;\bm\chi_\ell\in\{0,1\}^{Q}&\forall \ell\in[L]\\
	&&	\bm S\bm{\xi}=\bm t\\
	&&	\displaystyle\xi_\ell=\mathbf v_{Q}^\top \bm{\chi}_\ell&\forall \ell\in[L]\\
	&&\bm\chi_\ell+\bm{\eta}_\ell=\mathbf e&\forall \ell\in[L].
	\end{array}
	\end{equation}
	We next define the expanded vector
	\begin{equation*}
	{{\xih}}=\begin{bmatrix}	
	\bm{\chi}_1^\top
	& \cdots & 
	\bm{\chi}_L^\top &  
	\bm{\eta}_1^\top & \cdots &\bm{\eta}_L^\top & \bm{\xi}^\top
	\end{bmatrix}^\top\in\RR_+^{\Kh}
	\end{equation*}
	that comprises all decision variables in \eqref{eq:quad_max_}. Together with the augmented parameters \eqref{eq:expanded_parameters}, we can reformulate \eqref{eq:quad_max_} concisely as the problem
	\begin{equation}
\label{eq:quad_max__}
\begin{array}{clll}
Z(\bm x)=&\displaystyle \sup&\norm{\bm \Ab (\bm x) \xih}^2 + \bm \bb (\bm x)^\top{\xih} + c (\bm x) \\
&\st&\displaystyle {{\xih}}\in\RR_+^{\Kh}\\
&&	\Sb{{\xih}}=\tb\\
&& \xi_{\ell}'\in\{0,1\}&\forall \ell\in \left[LQ\right].
\end{array}
\end{equation}
The mixed-binary quadratic program \eqref{eq:quad_max__} already has the desired standard form \eqref{eq:standard_QP} with inputs $P=\Kh$, $\bm Q=\Ab(\bm x)^\top\Ab(\bm x)$, $\bm r=\bb(\bm x)$, $\bm F=\Sb$, $\bm g=\bm t$, and $\mathcal L=[LQ]$. We may thus apply Theorem \ref{thm:burer} to obtain the equivalent completely positive program \eqref{eq:CPP}. This completes the proof.
\end{proof}
We  remark that in view of the concise representation in Lemma \ref{lem:integer_binary}, the size of the completely positive program \eqref{eq:CPP} remains polynomial in the size of the input data. This completely positive program admits a dual copositive program given by
\begin{equation}
\label{eq:COP}
\begin{array}{clll}
\displaystyle \overline{Z}(\bm x) = &\inf &  c(\bm x) + \tb ^\top \bm \psi + (\tb \circ \tb) ^\top \bm \phi + \tau \\
&\st&\displaystyle \tau \in \RR ,\; \bm \psi,\bm \phi \in\RR^{\Jh},\;\bm \gamma \in\RR^{LQ} \\
&&\begin{bmatrix} \Sb^\top \diag(\bm\phi) \Sb - \Ab(\bm{x}) ^\top \Ab(\bm x) - \diag\left([\bm{\gamma}^\top\;\bm 0^\top]^\top \right)& \frac{1}{2}\left(\Sb^\top \bm\psi - \bb( \bm x) +[\bm{\gamma}^\top\;\bm 0^\top]^\top\right)
\\
\frac{1}{2}\left(\Sb^\top \bm\psi- \bb( \bm x) +[\bm{\gamma}^\top\;\bm 0^\top]^\top\right)^\top & \tau
\end{bmatrix} \succeq_{\mathcal C} \bm 0.
\end{array}
\end{equation}
By weak conic duality, the optimal value of this copositive program constitutes an upper bound on $Z(\bm x)$. 
\begin{prop} For any fixed decision $\bm x\in\mathcal X$ we have
$\overline{Z}(\bm x)\geq Z(\bm x)$.
\end{prop}

\subsection{A Copositive Reformulation of RQP}
In this section, we demonstrate strong duality for the primal and dual pair \eqref{eq:CPP} and \eqref{eq:COP}, respectively, under the natural boundedness assumption on the uncertainty set $\Xi$. This exactness result enables us to reformulate the RQP~\eqref{eq:RO} equivalently as a copositive program of polynomial size. 
\begin{thm}[Strong Duality]
	\label{thm:strong_duality}
	For any fixed decision $\bm x\in\mathcal X$ we have
	$\overline{Z}(\bm x)=Z(\bm x)$.
\end{thm}
We would like to mention that a similar result is proved in a recent paper (Theorem 8 in \cite{bomze2017fresh}). However, the two proofs are quite different from one another. While the proof in \cite{bomze2017fresh} establishes strong duality by proving the existence of a Slater point for a general copositive program, we show explicitly how to construct a Slater point for the copositive program \eqref{eq:COP} from input parameters. Because of its constructive nature, we believe our proof is interesting on its own and sheds some light on the geometry of the feasible region of~\eqref{eq:COP}.

We note  that the primal completely positive program \eqref{eq:CPP} never has an interior \cite{burer2012copositive}. In order to prove Theorem \ref{thm:strong_duality}, we  construct a Slater point for the dual copositive program \eqref{eq:COP}. The construction of the Slater point for problem \eqref{eq:COP} relies on the following two lemmas. We observe that by construction the boundedness of the uncertainty set $\Xi$ means that the lifted polytope
\begin{equation}
\label{eq:expanded_polytope_}
{\Xih}=\{{{\xih}}\in\RR^{\Kh}:\Sb{\xih}=\tb,\;{{\xih}}\geq\bm 0\}
\end{equation}
 is also bounded. This gives rise to the following lemma on the strict copositivity of the matrix $\Sb^\top\Sb$. 
\begin{lem}
	\label{lem:SS_copositive}
	We have
	$\Sb^\top\Sb\succ_{\mathcal C}\bm 0$.
\end{lem}
\begin{proof}
The boundedness assumption implies that the recession cone of the set $\Xih$ coincides with the point $\bm 0$, that is,
$\{{{\xih}}\in\RR_+^{\Kh}:\Sb{\xih}=\bm 0\}=\{\bm 0\}$. 
Thus, for every ${{\xih}}\geq \bm 0$, ${{\xih}}\neq \bm 0$, we must have $\Sb{{\xih}}\neq\bm 0$, which further implies that ${{\xih}}^\top\Sb^\top\Sb{{\xih}}>0$ for all ${{\xih}}\geq \bm 0$ such that ${{\xih}}\neq \bm 0$. Hence, the matrix $\Sb^\top\Sb$ is strictly copositive. 
\end{proof}
The next lemma, which was proven in \cite[Lemma 4]{hanasusanto2018conic}, constitutes an extension of the Schur complements lemma for matrices with a copositive sub-matrix.  We include the proof here to keep the paper self-contained. 
\begin{lem}[Copositive Schur Complements]
	\label{lem:copositive_schur}
	Consider the symmetric matrix
	\begin{equation*}
	\bm M=\begin{bmatrix}
	\bm P& \bm Q\\
	\bm Q^\top & \bm R
	\end{bmatrix}.
	\end{equation*}
	We then have $\bm M\succ_{\mathcal C}\bm 0$ if $\bm R-\bm Q^\top\bm P^{-1}\bm Q\succ_{\mathcal C}\bm 0$ and $\bm P\succ\bm 0$. 
\end{lem}
\begin{proof}
		Consider a non-negative vector $[\bm\xi^\top\;\bm\rho^\top]^\top\in\RR_+^{P+Q}$ satisfying $\mathbf e^\top\bm\xi+\mathbf e^\top\bm\rho=1$. We have
	\begin{equation*}
	\begin{array}{rll}
	[\bm\xi^\top\;\bm\rho^\top]\bm M[\bm\xi^\top\;\bm\rho^\top]^\top&=&\displaystyle\bm\xi^\top\bm P\bm\xi+2\bm\xi^\top\bm Q\bm\rho+\bm\rho^\top\bm R\bm\rho\\
	&=&\displaystyle(\bm\xi+\bm P^{-1}\bm Q\bm\rho)^\top\bm P(\bm\xi+\bm P^{-1}\bm Q\bm\rho)+\bm\rho^\top(\bm R-\bm Q^\top\bm P^{-1}\bm Q)\bm\rho\ \geq\ 0.
	\end{array}
	\end{equation*}
	The final inequality follows from the assumptions $\bm P\succ\bm 0$, $\bm R-\bm Q^\top\bm P^{-1}\bm Q\succ_{\mathcal C}\bm 0$ and $\bm \rho\geq \bm 0$. In fact, the inequality will be strict, which can be shown by considering the following two cases:
	\begin{enumerate}
		\item If $\bm\rho=\bm 0$, then $\mathbf e^\top\bm\xi=1$. Therefore $\bm\xi\neq \bm 0$, which implies that $(\bm\xi+\bm P^{-1}\bm Q\bm\rho)^\top\bm P(\bm\xi+\bm P^{-1}\bm Q\bm\rho)>0$.
		\item If $\bm\rho\neq\bm 0$, then the assumption $\bm R-\bm Q^\top\bm P^{-1}\bm Q\succ_{\mathcal C}\bm 0$ implies that $\bm\rho^\top(\bm R-\bm Q^\top\bm P^{-1}\bm Q)\bm\rho>0$.
	\end{enumerate}
	Therefore, in both cases, by rescaling we have $[\bm\xi^\top\;\bm\rho^\top]\bm M[\bm\xi^\top\;\bm\rho^\top]^\top > 0$ for all $[\bm\xi^\top\;\bm\rho^\top]^\top \in\RR_+^{P+Q}$ such that $[\bm\xi^\top\;\bm\rho^\top]^\top\neq\bm 0$. Hence, $\bm M\succ_{\mathcal C}\bm 0$. 	
\end{proof}

Using Lemmas \ref{lem:SS_copositive} and \ref{lem:copositive_schur}, we are now ready to prove the main strong duality result. 
\begin{proof}[Proof of Theorem \ref{thm:strong_duality}]
	We construct a Slater point $(\tau,\bm{\psi},\bm{\phi},\bm{\gamma})$ for problem \eqref{eq:COP}. Specifically, we set $\bm{\gamma}=\bm 0$, $\bm{\psi}=\bm 0$, and $\bm{\phi}=\rho\mathbf e$ for some $\rho>0$. Problem \eqref{eq:COP} then admits a Slater point if there exist scalars $\rho,\tau>0$, such that
	\begin{equation}
	\label{eq:Slater_matrix}
	\begin{bmatrix} \rho\Sb^\top\Sb - \Ab(\bm{x}) ^\top \Ab(\bm x) & - \frac{1}{2}\bb( \bm x)
		\\
		- \frac{1}{2}\bb( \bm x)^\top & \tau
	\end{bmatrix}\succ_{\mathcal C}\bm 0. 
	\end{equation}
	Lemma \ref{lem:SS_copositive} implies  that for a sufficiently large $\rho$ the matrix $\rho\Sb^\top\Sb - \Ab(\bm{x}) ^\top \Ab(\bm x)$ is strictly copositive. Thus, we can choose a positive $\tau$ to ensure that
	\begin{equation*}
	\rho\Sb^\top\Sb - \Ab(\bm{x}) ^\top \Ab(\bm x) -\frac{1}{4\tau}\bb(\bm x)\bb(\bm x)^\top\succ_{\mathcal C}\bm 0.
	\end{equation*}
	Using Lemma \ref{lem:copositive_schur}, we may conclude that the strict copositivity constraint in \eqref{eq:Slater_matrix} is satisfied by the constructed solution $(\tau,\bm{\psi},\bm{\phi},\bm{\gamma})$. Thus,  problem \eqref{eq:COP} admits a Slater point and  strong duality  indeed holds for the primal and dual pair \eqref{eq:CPP} and \eqref{eq:COP}, respectively. 
\end{proof}
The exactness result portrayed in Theorem \ref{thm:strong_duality} enables us to derive the equivalent copositive programming reformulation for~\eqref{eq:RO}. 
\begin{thm}
	\label{thm:CP_RQCQP}
		The RQP \eqref{eq:RO} is equivalent to the following copositive program. 
		\begin{equation}
		\label{eq:RO_COP}
		\begin{array}{clll}
		\displaystyle\minimize&\displaystyle c(\bm x) + \tb ^\top \bm \psi + (\tb \circ \tb) ^\top \bm \phi + \tau\\
		\subjectto& \displaystyle {\bm x\in\mathcal X},\;
		\tau \in \RR ,\; \bm \psi,\bm \phi \in\RR^{\Jh},\;\bm \gamma \in\RR^{LQ},\;\bm H\in\mathbb S^{\Kh}_+\\
		& \begin{bmatrix}
		\mathbb I & \Ab(\bm{x}) \\
		\Ab(\bm{x})^\top & \bm H
		\end{bmatrix}\succeq\bm 0  \\
		&\begin{bmatrix} \Sb^\top \diag(\bm\phi) \Sb - \bm H - \diag\left([\bm{\gamma}^\top\;\bm 0^\top]^\top\right)& \frac{1}{2}\left(\Sb^\top \bm\psi - \bb( \bm x) +[\bm{\gamma}^\top\;\bm 0^\top]^\top \right)
		\\
		\frac{1}{2}\left(\Sb^\top \bm\psi- \bb( \bm x) +[\bm{\gamma}^\top\;\bm 0^\top]^\top \right)^\top & \tau
		\end{bmatrix} \succeq_{\mathcal C} \bm 0
		\end{array}
		\end{equation}
\end{thm}
The proof of Theorem \ref{thm:CP_RQCQP} relies on the following lemma, which linearizes the quadratic term $\Ab(\bm{x}) ^\top \Ab(\bm x)$ in the left-hand side matrix of problem \eqref{eq:COP}. 
\begin{lem}
	\label{lem:linearize_A'A}
	Let $\bm M\in\mathbb S^R$ be a symmetric matrix  and $\bm A\in\RR^{P\times Q}$ be an arbitrary matrix with $Q\leq R$. Then the copositive inequality
	\begin{equation}
	\label{eq:copositive_A'A}
	\bm M\succeq_{\mathcal C}\begin{bmatrix}
	\bm A^\top\bm A& \bm 0\\
	\bm 0 & \bm 0
	\end{bmatrix}
	\end{equation}
	is satisfied if and only if there exists a positive semidefinite matrix $\bm H\in\mathbb S_+^Q$ such that 
	\begin{equation}
	\label{eq:inequalities_for_H}
	\bm M\succeq_{\mathcal C}\begin{bmatrix}
	\bm H& \bm 0\\
	\bm 0 & \bm 0
	\end{bmatrix}\quad\text{ and }\quad\begin{bmatrix}
	\mathbb I& \bm A\\
	\bm A^\top & \bm H
	\end{bmatrix}\succeq\bm 0. 
	\end{equation}	
\end{lem}
\begin{proof}
	The only if statement is satisfied immediately by setting $\bm H=\bm A^\top\bm A$. To prove the converse statement, assume that there exists such a positive semidefinite matrix $\bm H\in\mathbb S_+^Q$. Then by the Schur complement the semidefinite inequality in \eqref{eq:inequalities_for_H} implies that $\bm H\succeq\bm A^\top\bm A$ and, \emph{a fortiori}, $\bm H\succeq_{\mathcal C}\bm A^\top\bm A$. Combining this with the copositive inequality in \eqref{eq:inequalities_for_H} then yields \eqref{eq:copositive_A'A}. Thus, the claim follows. 
\end{proof}

\begin{proof}[Proof of Theorem \ref{thm:CP_RQCQP}]
Applying Theorem \ref{thm:strong_duality}, we may replace the objective function of \eqref{eq:RO} with the corresponding copositive reformulation, we thus find that problem \eqref{eq:RO} is equivalent to
	\begin{equation*}
	\begin{array}{clll}
		\displaystyle\minimize&\displaystyle c(\bm x) + \tb ^\top \bm \psi + (\tb \circ \tb) ^\top \bm \phi + \tau\\
		\subjectto& \displaystyle {\bm x\in\mathcal X},\;
		\tau \in \RR ,\; \bm \psi,\bm \phi \in\RR^{\Jh},\;\bm \gamma \in\RR^{LQ}\\
		&\begin{bmatrix} \Sb^\top \diag(\bm\phi) \Sb - \Ab(\bm{x}) ^\top \Ab(\bm x) - \diag\left([\bm{\gamma}^\top\;\bm 0^\top]^\top \right)& \frac{1}{2}\left(\Sb^\top \bm\psi - \bb( \bm x) +[\bm{\gamma}^\top\;\bm 0^\top]^\top \right)
			\\
			\frac{1}{2}\left(\Sb^\top \bm\psi- \bb( \bm x) +[\bm{\gamma}^\top\;\bm 0^\top]^\top\right)^\top & \tau
		\end{bmatrix} \succeq_{\mathcal C} \bm 0. &
	\end{array}
\end{equation*}
Next, we apply Lemma \ref{lem:linearize_A'A} to linearize the quadratic terms $\Ab(\bm{x}) ^\top \Ab(\bm x)$, which gives rise to the desired copositive program \eqref{eq:RO_COP}. This completes the proof. 
\end{proof}

\section{Conservative Semidefinite Programming Approximation}
\label{sec:SDP_formulation}
The copositive program~\eqref{eq:RO_COP} is intractable due to its equivalence with generic RQPs over a polyhedral uncertainty set~\cite{ben1998robust}. In the copositive reformulation, however, all the difficulty of the original problem \eqref{eq:RO} is shifted into the copositive cone $\mathcal C$, which has been well-studied in the literature. Specifically, there exists a hierarchy of 
increasingly tight semidefinite representable inner approximations that converge in finitely many iterations to $\mathcal C$  \cite{parrilo2000structured, bomze2002solving, DKP02:copositive,lasserre2009convexity}. The simplest of these approximations is given by the cone
\begin{equation*}
\mathcal C^0=\left\{\bm M\in\mathbb S^K:\bm M=\bm P+\bm N,\;\bm P\succeq\bm 0,\;\bm N\geq \bm 0 \right\},
\end{equation*}
which contains all symmetric matrices that can be decomposed into a sum of positive semidefinite and non-negative matrices. 
For dimensions $K\leq 4$ it can be shown that $\mathcal C^0=\mathcal C$  \cite{diananda62:copositive}, while for $K>4$, $C^0$ is a strict subset of $\mathcal C$. 

Replacing the cone $\mathcal C$ in \eqref{eq:RO_COP} with the inner approximation $\mathcal C^0$ gives rise to a tractable conservative approximation for the RQP~\eqref{eq:RO}. In this case, however, the resulting optimization problem might have no interior or even become infeasible as the Slater point constructed in Theorem \ref{thm:strong_duality} can fail to be a Slater point to the restricted problem.  Indeed, the strict copositivity of the matrix $\Sb^\top\Sb$ is in general insufficient to ensure that the matrix is also strictly positive definite. To remedy this shortcoming, we suggest the following simple modification to the primal completely positive formulation of $Z(\bm x)$ in \eqref{eq:CPP}. Specifically,  we assume that there exists a non-degenerate ellipsoid centered at $\bm c\in\RR_+^{\Kh}$ with radius $r\in\RR_{++}$ and shape parameter $\bm Q\in\mathbb S_{++}^{\Kh}$ given by \begin{equation*}
\mathcal B(r,\bm Q,\bm c)=\left\{{{\xih}}\in\RR^{\Kh}:\|\bm Q({{\xih}}-\bm c)\|\leq r\right\}
\end{equation*} that contains the lifted set $\Xih$ in \eqref{eq:expanded_polytope_}.
We then consider the following augmented completely positive programming reformulation for the maximization problem \eqref{eq:quadratic_program}.
\begin{equation}
\label{eq:CPP_ellipsoid}
\begin{array}{clll}
\displaystyle  Z(\bm x) = &\displaystyle \sup&\tr\left(\Ab(\bm x) \bm \Omega \Ab(\bm x) ^\top\right) + \bb(\bm x)^\top{ \xih} + c (\bm x)  \\
&\st&\displaystyle {{\xih}}\in\RR_+^{\Kh},\;\bm{\Omega} \in\mathbb S_+^{\Kh}\\
&& \Sb { \xih} = \tb,\;\diag(\Sb\bm \Omega\Sb^\top)=\tb\circ\tb\\
&& {\xi}_{\ell}'=\Omega_{\ell \ell} \quad \forall \ell\in\left[LQ\right]\\
&& \tr\left(\bm Q \bm \Omega \bm Q ^\top\right) - 2\bm c^\top\bm Q^\top\bm Q{{\xih}} + \bm c^\top\bm Q^\top\bm Q\bm c\leq r^2\\
&&\begin{bmatrix}  \bm\Omega & {\xih}\\
{ \xih}^\top & 1\\
\end{bmatrix}\succeq_{\mathcal C^*} \bm 0
\end{array}
\end{equation}
Here, we have added the redundant constraint $\tr\left(\bm Q \bm \Omega \bm Q ^\top\right) - 2\bm c^\top\bm Q^\top\bm Q{{\xih}} + \bm c^\top\bm Q^\top\bm Q\bm c\leq r^2$ to \eqref{eq:CPP}, which arises from linearizing the quadratic constraint 
\begin{equation*}
\|\bm Q({{\xih}}-\bm c)\|^2=\tr\left(\bm Q { \xih}{{\xih}}^\top \bm Q ^\top\right) - 2\bm c^\top\bm Q^\top\bm Q{{\xih}} + \bm c^\top\bm Q^\top\bm Q\bm c\leq r^2,
\end{equation*}
where we have set $\bm \Omega={ \xih}{{\xih}}^\top$.
The dual of the augmented problem \eqref{eq:CPP_ellipsoid} is given by the following copositive program.
\begin{equation}
\label{eq:COP_ellipsoid}
\begin{array}{clll}
\overline{Z}(\bm x)=&\inf & \displaystyle c(\bm x) + \tb ^\top \bm \psi + (\tb \circ \tb) ^\top \bm \phi +\lambda r^2-\lambda\|\bm Q\bm c\|^2+ \tau \\
&\st&\displaystyle \tau \in \RR ,\; \lambda\in\RR_+,\;\bm \psi,\bm \phi \in\RR^{\Jh},\;\bm \gamma \in\RR^{LQ},\;\bm h\in\RR^\Kh \\
&&\begin{bmatrix} \lambda\bm Q^\top\bm Q+\Sb^\top \diag(\bm\phi) \Sb - \Ab(\bm{x}) ^\top \Ab(\bm x) - \diag\left([\bm{\gamma}^\top\;\bm 0^\top]^\top\right)& \frac{1}{2}\bm h\\
\frac{1}{2}\bm h^\top & \tau \end{bmatrix} \succeq_{\mathcal C} \bm 0\\
&&\bm h=\Sb^\top \bm\psi - \bb( \bm x) +[\bm{\gamma}^\top\;\bm 0^\top]^\top -2\lambda\bm Q^\top\bm Q\bm c
\end{array}
\end{equation}
Note that we have $Z(\bm x)=\overline{Z}(\bm x)$ since all the new additional terms are redundant for the original reformulations. Nevertheless, since the ellipsoid $\mathcal B(r,\bm Q,\bm c)$ is non-degenerate, we find that the matrix $\bm Q^\top\bm Q$ is positive definite. We can thus set all eigenvalues of the scaled matrix $\lambda\bm Q^\top\bm Q$ to any arbitrarily large positive values by controlling the scalar $\lambda\in\RR_+$. This suggests that  replacing the cone $\mathcal C$ with its inner approximation $\mathcal C^0$ in~\eqref{eq:COP_ellipsoid} will always yield a problem with a Slater point. 

Apart from helping us prove the existence of a Slater point, adding an ellipsoidal constraint to the description of the uncertainty set can also be of help numerically. Although, the constraint is redundant for the exact problem, it might not be redundant for the conservative approximation obtained by replacing $\mathcal C$ with $\mathcal C^0$. Adding the constraint results in an additional variable $\lambda$ in the SDP approximation, which can improve the objective value. Ideally, we would like the volume of the ellipsoid to be as small as possible to get more improvement. However, determining the parameters of the ellipsoid having minimum volume that encloses the set ${\Xi}$ is NP-hard. A feasible ellipsoid that can be generated tractably is \mbox{$\{\bm\xi\in\RR^{K}: \|\bm\xi\|\leq \norm{\bm r}\}$}, where
\begin{equation*}
r_k = \sup_{\bm\xi\in{\Xi}}\;{\xi}_k, \quad \forall k \in [K].
\end{equation*}
Note that the parameter $\bm r$ of the ellipsoid can be determined by solving $K$ linear programs. Depending on the specific uncertainty set at hand, it might be possible to find other tighter ellipsoidal approximations. 

\subsection{Comparison with the Approximate \slemma Method}
\label{sec:S-lemma}
Next, we show that solving the problem by replacing $\mathcal C$ with the simplest inner approximation $\mathcal C^0$ is better than the approximate \slemma method. Since the latter is only valid in the case of continuous uncertain parameters, we restrict the discussion to the case where the bounded uncertainty set contains no integral terms and is given by the polytope
$\Xi=\left\{\bm \xi\in\RR_+^K: \bm S\bm \xi= \bm t \right\}$.
Here, the extended parameters \eqref{eq:expanded_parameters} simplify to
\begin{equation*}
\label{eq:expanded_parameters_polytope}
\begin{array}{c}
\Sb=
\bm S,\quad 
\tb=
\bm t,
\quad\Ab(\bm x)=
\bm A(\bm x) 
,\quad \text{and}\quad\bb(\bm x)=
\bm b(\bm x),
\end{array}
\end{equation*}
while the maximization problem \eqref{eq:quadratic_program} reduces to
\begin{equation}
\label{eq:quad_max_over_polytope}
{Z}(\bm x) =\sup_{{{\bm\xi}}\in\Xi}\norm{\bm A(\bm x)\bm\xi}^2 + \bm b (\bm x)^\top{ \bm\xi} + c (\bm x).
\end{equation}
The copositive programming reformulation \eqref{eq:COP_ellipsoid} can then be simplified to
\begin{equation}
\label{eq:COP_polytope}
\begin{array}{clll}
\displaystyle \overline{Z}(\bm x) = &\inf &  c(\bm x) + \bm t ^\top \bm \psi + (\bm t \circ \bm t) ^\top \bm \phi +\lambda r^2-\lambda\|\bm Q\bm c\|^2+ \tau\\
&\st&\displaystyle \tau \in \RR ,\;\lambda\in\RR_+ ,\;\bm \psi,\bm \phi \in\RR^{J}\\
&&\begin{bmatrix} \lambda\bm Q^\top\bm Q+\bm S^\top \diag(\bm\phi) \bm S - \bm A(\bm{x}) ^\top \bm A(\bm x) & \frac{1}{2}\left(\bm S^\top \bm\psi - \bm b(\bm x) -2\lambda\bm Q^\top\bm Q\bm c \right)
\\
\frac{1}{2}\left(\bm S^\top \bm\psi- \bm b( \bm x) -2\lambda\bm Q^\top\bm Q\bm c\right)^\top & \tau
\end{bmatrix} \succeq_{\mathcal C} \bm 0.
\end{array}
\end{equation}
Replacing the cone $\mathcal C$ in \eqref{eq:COP_polytope} with its inner approximation $\mathcal C^0$, we obtain a tractable SDP reformulation whose optimal value $\overline{Z}^{\mathcal C_0}(\bm x)$  constitutes an  upper bound on ${Z}(\bm x)$. Alternatively, we describe the approximate \slemma method below, which provides a different conservative SDP approximation for \eqref{eq:quad_max_over_polytope}.
\begin{prop}[Approximate \slemma Method \cite{ben1998robust}]
	Assume that the uncertainty set is a bounded polytope  and  there is an ellipsoid centered at $\bm c\in\RR_+^{K}$ of radius $r$ given by $\mathcal B(r,\bm Q,\bm c)=\{{{\bm\xi}}\in\RR^{K}:\|\bm Q({{\bm\xi}}-\bm c)\|\leq r\}$ that contains the set~${\Xi}$. Then, for any fixed $\bm x\in\mathcal X$, the maximization problem \eqref{eq:quadratic_program} is upper bounded by the optimal value of the following semidefinite program:
	\begin{equation}
	\label{eq:S_lemma}
	\begin{array}{rlll}
	\overline Z^\mathcal S(\bm x)=&\displaystyle \inf&\displaystyle\; c (\bm x) +\bm t^\top\bm \theta+\rho r^2-\rho \|\bm Q\bm{c}\|^2+\kappa\\
	&\displaystyle\st & \kappa\in\RR,\;\rho\in\RR_+,\;\bm{\theta}\in\RR^J,\;\bm{\eta}\in\RR_+^{J}\\
	&&\begin{bmatrix}\rho\bm Q^\top\bm Q- \bm A(\bm{x}) ^\top \bm A(\bm x) & \frac{1}{2}\left(\bm S^\top \bm\theta - \bm b( \bm x) -\bm{\eta}-2\rho\bm Q^
	\top\bm Q\bm c \right)
	\\
	\frac{1}{2}\left(\bm S^\top \bm\theta - \bm b( \bm x) -\bm{\eta}-2\rho\bm Q^
	\top\bm Q\bm c\right)^\top & \kappa
	\end{bmatrix} \succeq \bm 0. \\
	\end{array}
	\end{equation}
\end{prop}
\begin{proof}
	The quadratic maximization problem in \eqref{eq:quad_max_over_polytope} can be equivalently reformulated as
	\begin{equation*}
	\begin{array}{clll}
	Z(\bm x)=&\displaystyle \sup&\norm{\bm A(\bm x) \bm\xi}^2 + \bm b(\bm x)^\top{\bm\xi} + c (\bm x) \\
	&\st&\displaystyle {{\bm{\xi}}}\in\RR_+^{K}\\
	&&	\bm S \bm\xi = \bm t\\
	&& \|\bm Q({{\bm{\xi}}}-\bm c)\|^2\leq r^2. 
	\end{array}
	\end{equation*}
	Here, the last constraint is added without loss generality since $\Xi\subseteq\mathcal B(r,\bm Q,\bm c)$. 
	Reformulating the problem into its Lagrangian form then yields
	\begin{equation*}
	\begin{array}{rcll}
	&Z(\bm x)\\
	=&\displaystyle \sup_{\bm\xi} &\displaystyle\inf_{\bm{\eta}\geq\bm 0,\rho\geq 0,\bm{\theta}}\;\norm{\bm A(\bm x) \bm{\xi}}^2 + \bm b(\bm x)^\top{ \bm{\xi}} + c (\bm x) + \bm t^\top\bm \theta - \bm{\xi}^\top\bm S^\top\bm \theta+\bm{\xi}^\top\bm{\eta}+\rho r^2-\rho \|\bm Q({{\bm{\xi}}}-\bm c)\|^2\\
	\leq &\displaystyle\inf_{\bm{\eta}\geq\bm 0,\rho\geq 0,\bm{\theta}}&\displaystyle\; \quad\sup_{\bm{\xi}}\quad\;\;\norm{\bm A (\bm x) { \bm{\xi}}}^2 + \bm b (\bm x)^\top \bm{\xi} + c (\bm x) +\bm t^\top\bm \theta -\bm{\xi}^\top\bm S^\top\bm \theta+\bm{\xi}^\top\bm{\eta}+\rho r^2-\rho \|\bm Q(\bm{\xi}-\bm c)\|^2\\
	=&\displaystyle\inf_{\bm{\eta}\geq\bm 0,\rho\geq 0,\bm{\theta}}&\displaystyle\; c (\bm x) +\bm t^\top\bm \theta+\rho r^2-\rho\|\bm Q\bm{c}\|^2\\
	&&\qquad\qquad\displaystyle+\sup_{\bm{\xi}}\left(\norm{\bm A(\bm x) \bm{\xi}}^2 + \bm b(\bm x)^\top\bm{\xi} -\bm\xi^\top\bm S^\top\bm \theta+ \bm\xi^\top\bm{\eta}-\rho \|\bm Q{{\bm{\xi}}}\|^2+2\rho\bm{\xi}^\top\bm Q^\top\bm Q\bm c\right),
	\end{array}
	\end{equation*}
	where the inequality follows from the weak Lagrangian duality. We next introduce an epigraphical variable~$\kappa$ that shifts the supremum in the objective function into the constraint. We have
	\begin{equation*}
	\begin{array}{rlll}
	Z(\bm x)\leq&\displaystyle \inf&\displaystyle c (\bm x) +\bm t^\top\bm \theta+\rho r^2-\rho\|\bm Q\bm{c}\|^2+\kappa\\
	&\displaystyle\st & \bm{\theta}\in\RR^J,\;\bm{\eta}\in\RR_+^{K},\;\rho\in\RR_+,\;\kappa\in\RR\\
	&&\displaystyle\sup_{\bm \xi}\left(\norm{\bm A(\bm x) \bm{\xi}}^2 + \bm b (\bm x)^\top{ \bm{\xi}} -{{\bm{\xi}}}^\top\bm S^\top\bm \theta+ \bm{\xi}^\top\bm{\eta}-\rho \|\bm Q\bm{\xi}\|^2+2\rho\bm{\xi}^\top\bm Q^\top\bm Q\bm c\right)\leq \kappa. 
	\end{array}
	\end{equation*}
	Reformulating the semi-infinite constraint as a semidefinite constraint then yields the desired reformulation~\eqref{eq:S_lemma}. This completes the proof. 
\end{proof}

The next proposition shows that the approximation resulting from replacing the copositive cone $\mathcal C$  in \eqref{eq:COP_polytope} with its coarsest inner approximation $\mathcal C^0$ is stronger than the  state-of-art approximate \slemma method. 
\begin{prop}
	\label{prop:vs_S_lemma}
	The following relation holds.
	\begin{equation*}
	{Z}(\bm x)=\overline{Z}(\bm x)\leq\overline{Z}^{\mathcal C_0}(\bm x)\leq\overline{Z}^{\mathcal S}(\bm x)
	\end{equation*}
\end{prop}
\begin{proof}
	The equality and the first inequality hold by construction. To prove the second inequality, we consider the following semidefinite program that arises from replacing the cone $\mathcal C$ with the inner approximation~$\mathcal C^0$ in~\eqref{eq:COP_polytope}. 
	\begin{equation}
	\label{eq:COP_polytope_inner}
	\begin{array}{clll}
	\displaystyle \overline{Z}^{\mathcal C_0}(\bm x) = &\inf &  c(\bm x) + \bm t^\top \bm \psi + (\bm t \circ \bm t) ^\top \bm \phi +\lambda r^2-\lambda\|\bm Q\bm c\|^2+ \tau\\
	&\st&\displaystyle \tau \in \RR ,\; ,\;\lambda,h\in\RR_+ ,\;\bm \psi,\bm \phi \in\RR^{J},\;\bm F\in\RR_+^{K\times K},\;\bm g\in\RR_+^{K}\\
	&&\begin{bmatrix} \lambda\bm Q^\top\bm Q+\bm S^\top \diag(\bm\phi) \bm S - \bm A(\bm{x}) ^\top \bm A(\bm x) & \frac{1}{2}\left(\bm S^\top \bm\psi - \bm b(\bm x) -2\lambda\bm Q^\top\bm Q\bm c \right)
	\\
	\frac{1}{2}\left(\bm S^\top \bm\psi- \bm b(\bm x) -2\lambda\bm Q^\top\bm Q\bm c\right)^\top & \tau
	\end{bmatrix} \succeq\begin{bmatrix}\bm F&\bm g\\\bm g^\top & h
	\end{bmatrix}
	\end{array}
	\end{equation}
	Next, we show that any feasible solution $(\kappa,\rho,\bm{\theta},\bm{\eta})$ to \eqref{eq:S_lemma} can be used to construct a feasible solution $(\tau,\lambda,h,\bm \psi,\bm \phi,\bm F,\;\bm g)$ to \eqref{eq:COP_polytope_inner} with the same objective value. Specifically, we set 
	$\tau=\kappa$, $\lambda=\rho$,  $h=0$, $\bm \psi=\bm{\theta}$, $\bm{\phi}=\bm{0}$, $\bm F=\bm 0$, and $\bm g=\bm{\eta}$. The feasibility of the solution $(\kappa,\rho,\bm{\theta},\bm{\eta})$  in \eqref{eq:S_lemma} then implies that the constructed solution  $(\tau,\lambda,h,\bm \psi,\bm \phi,\bm F,\;\bm g)$ is also feasible in \eqref{eq:COP_polytope_inner}. One can verify that these solutions give rise to the same objective function value for the respective problems. Thus, the claim follows. 
\end{proof}

Next, we demonstrate that the inequality in $\overline{Z}^{0}(\bm x)\leq\overline{Z}^{\mathcal S}(\bm x)$ in Proposition \ref{prop:vs_S_lemma} can often be strict. This affirms that the proposed  SDP approximation \eqref{eq:COP_polytope_inner} is indeed stronger than the approximate \slemma method. 

\begin{ex} 
	Consider the following quadratic maximization problem:
	\begin{equation}
	\label{eq:simple_quad_exa}
	\begin{array}{clll}
	\displaystyle Z(\bm x) = &\sup &  \xi_{1}^2 \\
	&\st&\displaystyle  \bm \xi \in \RR_+^2\\
	&& 2\xi_{1} + \xi_{2} = 2.
	\end{array}
	\end{equation}
	A simple analysis shows that $Z(\bm x)=1$, which is attained at the solution $(\xi_1,\xi_2)=(1,0)$.
	The problem \eqref{eq:simple_quad_exa} constitutes an instance of problem \eqref{eq:quad_max_over_polytope} with the parameterizations
	\begin{equation*}
	\bm A(\bm x)=\begin{bmatrix}
	1 & 0 
	\end{bmatrix},\quad \bm b(\bm x)=\bm 0,\quad \textup{and}\quad c(\bm x)= 0.
	\end{equation*}
	Here, the uncertainty set is given by the polytope $\Xi=\{\bm{\xi}\in\RR_+^2:2\xi_1+\xi_2=2\}$, which corresponds to the inputs $\bm S=[2\;1]$ and $\bm t=2$.  Replacing the cone $\mathcal C$ with its inner approximation $\mathcal C^0$ in the copositive programming reformulation of \eqref{eq:simple_quad_exa}, we find that the resulting semidefinite program yields the same optimal objective value of $\overline Z^{\mathcal C_0}(\bm x)=1$. Meanwhile, the corresponding approximate \slemma method yields an optimal objective value $Z^{\mathcal S}(\bm x)=4$. Thus, while the SDP approximation of the copositive program~\eqref{eq:COP_polytope} is tight, the approximate \slemma generates an inferior objective value for the simple instance~\eqref{eq:simple_quad_exa}. 
\end{ex}

\section{Extensions}
\label{sec:extensions}
In this section, we discuss several extensions to the RQP \eqref{eq:RO} which are also amenable to exact copositive programming reformulation. In Section \ref{extend:two_stage}, we study two-stage robust optimization with mixed-integer uncertainty set where the objective is quadratic in the first- and the second-stage decision variables. In Section \ref{extend:rqcqp}, we develop an extension to the case when the model has robust quadratic constraints. Finally, in Section \ref{extend:nonconvex_obj}, we discuss the case where the objective function contains quadratic terms which are not convex in the uncertain parameter vector $\bm \xi$.

\subsection{Two-Stage Robust Quadratic Optimization}
\label{extend:two_stage}
In this section, we study the two-stage robust quadratic optimization problems of the form
\begin{equation}
\label{eq:RO_two_stage}
\begin{array}{clll}
\displaystyle\minimize&\displaystyle\sup_{\bm{\xi}\in\Xi}\norm{\bm A(\bm x) \bm \xi}^2 +  \bm b(\bm x)^\top\bm \xi + c (\bm x)+ \mathcal R(\bm x,\bm{\xi})\\
\subjectto& \displaystyle {\bm x\in\mathcal X}.
\end{array}
\end{equation}
Here, for any fixed decision $\bm x\in\mathcal X$ and uncertain parameter realization $\bm{\xi}\in\Xi$, the second-stage cost $\mathcal R(\bm x,\bm{\xi})$ coincides with the optimal value of the convex quadratic program given by
 \begin{equation}
 \label{eq:recourse}
 \begin{array}{clll}
\mathcal  R(\bm x,\bm{\xi})=& \displaystyle\inf&\displaystyle\|\bm P\bm y\|^2+\left(\bm R\bm{\xi}+\bm r\right)^\top\bm y\\
 &\st& \displaystyle \bm y\in\RR^{D_2}\\
 &&\displaystyle\bm T(\bm x)\bm{\xi}+\bm h(\bm x)\leq \bm W\bm y,
 \end{array}
 \end{equation}
 where $\bm T(\bm x):\mathcal X\rightarrow\RR^{T\times K}$ and $\bm h(\bm x):\mathcal X\rightarrow\RR^T$ are matrix- and vector-valued affine functions, respectively. 
 \begin{ex}[Support Vector Machines with Noisy Labels] Consider the following soft-margin support vector machines (SVM) model for data classification.
 \begin{equation}
 \label{eq:SVM}
 	\begin{array}{clll}
 \minimize & \displaystyle\lambda \|\bm w\|^2+\sum_{n\in[N]}\max\left\{0,1-\hat \xi_n(\bm w^\top\hat{\bm{\chi}_n}-w_0)\right\}\\
 \subjectto & \bm w\in\RR^K,\; w_0\in\RR
 \end{array}
 \end{equation}
 Here, for every index $n\in[N]$, the vector $\hat{\bm{\chi}}_n\in\RR^K$ is a data point that has been labeled as $\hat \xi_n\in\{-1,1\}$. The objective of  problem \eqref{eq:SVM} is to find a hyperplane $\{\bm{\chi}\in\RR^K:\bm w^\top\bm{\chi}=w_0\}$ 
 that separates all points labeled  $+1$ with the ones labeled $-1$. If the hyperplane satisfies $\hat \xi_n(\bm w^\top\hat{\bm{\chi}}_n-w_0)> 1$, $n\in[N]$, then the data points  are linearly separable. In practice, however, these data points may not be linearly separable. We thus seek the best linear separator that minimizes the number of incorrect classifications. This non-convex objective is captured by employing the hinge loss term $\sum_{n\in[N]}\max\left\{0,1-\hat \xi_n(\bm w^\top\hat{\bm{\chi}}_n-w_0)\right\}$  in  \eqref{eq:SVM} as a convex surrogate. Here, the term $\lambda\|\bm w\|^2$ in the objective function constitutes a regularizer for the coefficient~$\bm w$.  
 
 If the labels $\{\hat \xi_n\}_{n\in[N]}$ are erroneous, then one could envisage a robust optimization model that seeks the best linear separator in view of the most adverse realization of the labels. To this end, we assume that the vector of labels $\bm \xi$ is only known to  reside in a prescribed binary uncertainty set $\Xi\subseteq \{-1,1\}^N$. Then an SVM model that is robust against uncertainty in the labels can be formulated as
  \begin{equation*}
 \begin{array}{clll}
 \minimize & \displaystyle\lambda \|\bm w\|^2+\sup_{\bm{\xi}\in\Xi} \mathcal R(\bm w,w_0,\bm{\xi})\\ 
 \subjectto & \bm w\in\RR^K,\; w_0\in\RR,
 \end{array}
 \end{equation*}
 where 
 \begin{equation*}
 \begin{array}{clll}
 \mathcal  R(\bm w,w_0,\bm{\xi})=& \displaystyle\inf&\displaystyle\mathbf e^\top\bm y\\
 &\st& \displaystyle \bm y\in\RR_+^{N}\\
 &&\displaystyle y_n\geq 1-{\xi_n}(\bm w^\top\hat{\bm\chi}_n-w_0)&\forall n\in[N]. 
 \end{array}
 \end{equation*}
 This problem constitutes an instance of \eqref{eq:RO_two_stage} with the decision vector $\bm x=(\bm w,w_0)$, and the input parameters 
 \begin{equation*}
 \begin{array}{l}
 \bm A(\bm x)=\bm 0,\quad\bm b(\bm x)=\bm 0,\quad c(\bm x)=\lambda\|\bm w\|^2,\quad\bm P=\bm 0,\;\quad\bm R=\bm 0,\quad\bm r=\mathbf e,\\
\bm T(\bm x) = -\diag\left(\begin{bmatrix}
\bm w^\top\hat{\bm{\chi}}_1\\
\vdots\\
\bm w^\top\hat{\bm{\chi}}_N
\end{bmatrix}\right)-w_0\mathbb I,\quad\bm h(\bm x)=\mathbf e,\quad\text{and}\quad\bm W=\mathbb I. 
 \end{array}
 \end{equation*}
 \end{ex}

 The exactness result portrayed in Theorems \ref{thm:strong_duality} and \ref{thm:CP_RQCQP} can be extended to the two-stage robust optimization problem~\eqref{eq:RO_two_stage}. Specifically, if the problem has a \emph{complete recourse}\footnote{The two-stage problem \eqref{eq:RO_two_stage} has complete recourse if there exists $\bm y^+\in\RR^{D_2}$ with $\bm W\bm y^+>\bm 0$, which implies that the second-stage subproblem is feasible for every $\bm x\in\RR^{D_1}$ and $\bm{\xi}\in\RR^K$.} then, by employing Theorem \ref{thm:strong_duality} and extending the techniques developed in \cite[Theorem 4]{hanasusanto2018conic}, the two-stage problem \eqref{eq:RO_two_stage} can be reformulated as a copositive program of polynomial size.
 
 \begin{thm}
 	\label{thm:COP_two_stage}
 Assume that  $\bm P$ has full column rank.
  Then the two-stage robust optimization problem \eqref{eq:RO_two_stage} is equivalent to the copositive program
 \begin{equation}
 \label{eq:RO_COP_two_stage}
 \begin{array}{clll}
 \displaystyle\minimize&\displaystyle c(\bm x) -\frac{1}{4}\bm r^\top(\bm P^\top\bm P)^{-1}\bm r+ \tb ^\top \bm \psi + (\tb \circ \tb) ^\top \bm \phi + \tau\\
 \subjectto& \displaystyle {\bm x\in\mathcal X},\;
 \tau \in \RR ,\; \bm \psi,\bm \phi \in\RR^{\Jh},\;\bm \gamma \in\RR^{LQ},\;\bm H\in\mathbb S^{\Kh}_+ \\
 & \begin{bmatrix}
 \mathbb I & \Ab(\bm{x}) \\
 \Ab(\bm{x})^\top & \bm H
 \end{bmatrix}\succeq\bm 0 \\
 &\begin{bmatrix} \Sb^\top \diag(\bm\phi) \Sb - \bm H -\bm{\mathcal P}(\bm x)- \diag\left([\bm{\gamma}^\top\;\bm 0^\top]^\top\right)& \frac{1}{2}\left(\Sb^\top \bm\psi - \bb( \bm x) +[\bm{\gamma}^\top\;\bm 0^\top]^\top \right)
 \\[4mm]
 \frac{1}{2}\left(\Sb^\top \bm\psi- \bb( \bm x) +[\bm{\gamma}^\top\;\bm 0^\top]^\top\right)^\top & \tau
 \end{bmatrix} \succeq_{\mathcal C} \bm 0,
 \end{array}
 \end{equation}
 where 
 \begin{equation*}
 \begin{array}{c}
 \Sb=\begin{bmatrix}
 \bm 0 & \cdots & \bm 0 &  \bm 0 & \cdots & \bm 0 & \bm S &\bm 0\\
 -\mathbf v_{Q}^\top & \cdots & \bm 0^\top & \bm 0^\top & \cdots & \bm 0^\top & \mathbf e_1^\top &\bm 0^\top\\
 \vdots & \ddots & \vdots & \vdots & \ddots & \vdots & \vdots  & \vdots \\
 \bm 0^\top & \cdots &  -\mathbf v_{Q}^\top & \bm 0^\top & \cdots & \bm 0^\top & \mathbf e_L^\top &\bm 0^\top\\
 \mathbb I & \cdots & \bm 0 &\mathbb I & \cdots & \bm 0& \bm 0  &\bm 0\\
 \vdots & \ddots & \vdots & \vdots & \ddots & \vdots & \vdots &  \vdots \\
 \bm 0 & \cdots & \mathbb I &\bm 0 & \cdots & \mathbb I & \bm 0  &\bm 0\\
 \end{bmatrix}\in\RR^{\Jh\times \Kh},\quad 
 \tb=\begin{bmatrix}
 \bm t \\ 
 0\\
 \vdots \\
 0\\
 \mathbf e\\
 \vdots \\
 \mathbf e\\
 \end{bmatrix}\in\RR^{\Jh},\\
 \bm{\mathcal P}(\bm x)=\begin{bmatrix}
 \bm 0 & \cdots & \bm 0 &  \bm 0 & \cdots & \bm 0 & \bm 0 &  \bm 0\\
 \vdots & \ddots & \vdots & \vdots & \ddots & \vdots & \vdots  & \vdots \\
 \bm 0 & \cdots & \bm 0 &  \bm 0 & \cdots & \bm 0 & \bm 0  & \bm 0\\
 \bm 0 & \cdots & \bm 0 &  \bm 0 & \cdots & \bm 0 & -\frac{1}{4}\bm R^\top(\bm P^\top\bm P)^{-1}\bm R & \frac{1}{2}\left(\bm T(\bm x)+\frac{1}{2}\bm W(\bm P^\top\bm P)^{-1}\bm R\right)^\top\\
 \bm 0 & \cdots & \bm 0 &  \bm 0 & \cdots & \bm 0 & \frac{1}{2}\left(\bm T(\bm x)+\frac{1}{2}\bm W(\bm P^\top\bm P)^{-1}\bm R\right)& -\frac{1}{4}\bm W(\bm P^\top\bm P)^{-1}\bm W^\top\\
 \end{bmatrix}\in\mathbb S^{\Kh},
 \end{array}
 \end{equation*}
 \begin{equation*}
 \begin{array}{c}
 \quad\quad\Ab(\bm x)=\begin{bmatrix}
 \bm 0&\cdots & \bm 0 & \bm 0&\cdots & \bm 0 & \bm A(\bm x) & \bm 0
 \end{bmatrix}\in\RR^{M\times \Kh},\quad \text{ and }\\\bb(\bm x)=\begin{bmatrix}
 \bm 0^\top&\cdots & \bm 0^\top & \bm 0^\top&\cdots & \bm 0^\top & (\bm b(\bm x)-\frac{1}{2}\bm R^\top(\bm P^\top\bm P)^{-1}\bm r)^\top & (\bm h(\bm x)-\frac{1}{2}\bm W(\bm P^\top\bm P)^{-1}\bm r)^\top
 \end{bmatrix}^\top\in\RR^{\Kh},
 \end{array}
 \end{equation*}
 with
 \begin{equation*}
 \Jh=LQ+J+L\quad\text{and}\quad\Kh=2LQ+K+T.
 \end{equation*}
 \end{thm}
\begin{proof}
	Since $\bm P$ has full column rank, the matrix $\bm P^\top\bm{P}$ is positive definite. Thus, for any fixed $\bm x\in\mathcal X$ and $\bm{\xi}\in\Xi$, the recourse problem \eqref{eq:recourse} admits a dual quadratic program given by
 \begin{equation}
\label{eq:recourse_dual}
\begin{array}{clll}
\mathcal  R(\bm x,\bm{\xi})=& \displaystyle\sup&\displaystyle-\frac{1}{4}\left((\bm W^\top\bm{\theta}-\bm R\bm{\xi}-\bm r)^\top(\bm P^\top\bm P)^{-1}(\bm W^\top\bm{\theta}-\bm R\bm{\xi}-\bm r)\right)+\bm h(\bm x)^\top\bm{\theta}+\bm{\xi}^\top\bm T(\bm x)^\top\bm \theta\\
&\st& \displaystyle \bm \theta\in\RR_+^T.
\end{array}
\end{equation}
Strong duality holds as the two-stage problem \eqref{eq:RO_two_stage} has complete recourse. Substituting the dual formulation~\eqref{eq:recourse_dual} into the objective of \eqref{eq:RO_two_stage} yields
\begin{equation*}
\begin{array}{clll}
&\displaystyle\sup_{\bm{\xi}\in\Xi}\norm{\bm A (\bm x) \bm \xi}^2 +  \bm b(\bm x)^\top\bm \xi + c (\bm x)+ \mathcal R(\bm x,\bm{\xi})\\
=&\displaystyle\sup_{\bm{\xi}\in\Xi,\bm{\theta}\in\RR_+^T}\norm{\bm A(\bm x) \bm \xi}^2 +  \bm b(\bm x)^\top\bm \xi + c (\bm x)-\frac{1}{4}\left((\bm W^\top\bm{\theta}-\bm R\bm{\xi}-\bm r)^\top(\bm P^\top\bm P)^{-1}(\bm W^\top\bm{\theta}-\bm R\bm{\xi}-\bm r)\right)\\
&\quad\quad\quad\quad\quad\quad+\bm h(\bm x)^\top\bm{\theta}+\bm{\xi}^\top\bm T(\bm x)\bm \theta.
\end{array}
\end{equation*}	
Thus, for any fixed $\bm x\in\mathcal X$, the objective value of the two-stage problem \eqref{eq:RO_two_stage} coincides with the optimal value of a quadratic maximization problem, which is amenable to an exact completely positive programming reformulation similar to the one derived in Proposition \ref{prop:CPP}. We can then follow the same steps taken in the proofs of Theorems \ref{thm:strong_duality} and \ref{thm:CP_RQCQP} to obtain the equivalent copositive program \eqref{eq:RO_COP_two_stage}. This completes the proof. 
%
%
%
\end{proof}
 \begin{rem}
 The assumption that $\bm P$ has full column rank in Theorem \ref{thm:COP_two_stage} can be relaxed. If $\bm P$ does not have full column rank then the symmetric matrix  $\bm P^\top\bm P$ is not positive definite but admits the eigendecomposition 
 	$\bm P^\top\bm P=\bm U\bm \Lambda\bm U^{-1}$,
 	where $\bm U$ is an orthogonal matrix whose columns are the eigenvectors of $\bm P^\top\bm P$, while $\bm{\Lambda}$ is a diagonal matrix with the eigenvalues of   $\bm P^\top\bm P$ on its main diagonal. 
 		We assume without loss of generality that the matrix $\bm\Lambda$ has the block diagonal form
 		\begin{equation*}
 		\begin{bmatrix}
 			\bm\Lambda_+ & \bm 0\\
 			\bm 0 & \bm 0
 		\end{bmatrix},
 		\end{equation*}
 		where $\bm{\Lambda}_+$ is a diagonal matrix whose main diagonal comprises the non-zero eigenvalues of $\bm P^\top\bm P$. 
 		Next, by using the constructed eigendecomposition and performing the change of variable $\bm z\leftarrow \bm U^{-1}\bm y$, we can reformulate the recourse problem \eqref{eq:recourse} equivalently as 
 		 \begin{equation*}
 		\label{eq:recourse_}
 		\begin{array}{clll}
 		\mathcal  R(\bm x,\bm{\xi})=& \displaystyle\inf&\displaystyle\bm z_+^\top\bm\Lambda_+\bm z_++\left(\bm R\bm{\xi}+\bm r\right)^\top\bm U_+\bm z_++\left(\bm R\bm{\xi}+\bm r\right)^\top\bm U_0\bm z_0\\
 		&\st& \displaystyle (\bm z_+,\bm z_0)\in\RR^{D_2}\\
 		&&\displaystyle\bm T(\bm x)\bm{\xi}+\bm h(\bm x)\leq \bm W\bm U_+\bm z_++\bm W\bm U_0\bm z_0,
 		\end{array}
 		\end{equation*}
 		where $\bm U=[\bm U_+\;\;\bm U_0]$ and $\bm z = [\bm z_+^\top\;\;\bm z_0^\top]^\top$. The dual of this problem is given by the following quadratic program with a linear constraint system:
 \begin{equation*}
\label{eq:recourse_dual_}
\begin{array}{clll}
\mathcal  R(\bm x,\bm{\xi})=& \displaystyle\sup&\displaystyle-\frac{1}{4}\left((\bm W^\top\bm{\theta}-\bm R\bm{\xi}-\bm r)^\top\bm U_+^\top\bm\Lambda_+^{-1}\bm U_+(\bm W^\top\bm{\theta}-\bm R\bm{\xi}-\bm r)\right)+\bm h(\bm x)^\top\bm{\theta}+\bm{\xi}^\top\bm T(\bm x)^\top\bm \theta\\
&\st& \displaystyle \bm \theta\in\RR_+^T\\
& & \displaystyle \bm U_0^\top\left(\bm R\bm{\xi}+\bm r\right)=\bm U_0^\top\bm W^\top\bm \theta.
\end{array}
\end{equation*}
We can then repeat the same steps in the proof of Theorem \ref{thm:COP_two_stage} to obtain an equivalent copositive programming reformulation. We omit this result for the sake of brevity. 
 \end{rem}

\subsection{Robust Quadratically Constrained Quadratic Programming (RQCQP)}
\label{extend:rqcqp}
The setting that we consider can be extended to the case where, in addition to the robust quadratic objective function, there are several robust quadratic constraints of the form
\begin{equation}
	\label{eq:wc_cons}
	\sup_{\bm{\xi}\in\Xi}\left\{\norm{\bm A_i (\bm x) \bm \xi}^2 + \bm b_i (\bm x)^\top\bm \xi + c_i (\bm x)\right\} \leq 0 \qquad\forall i\in[I]. 
\end{equation}
In this case, the goal is to find a decision $\bm x \in \mathcal X$ which minimizes the worst-case objective function, while ensuring that the quadratic constraints are satisfied for all possible uncertain parameter vectors in $\Xi$. 

For every $ i \in [I]$, we define $\Ab_i(\bm x)$ and $\bb_i(\bm x)$ similarly to the definitions of the extended parameters $\Ab(\bm x)$ and $\bb(\bm x)$ in \eqref{eq:expanded_parameters}. By applying Theorem~\ref{thm:strong_duality}, the quadratic maximization problem in the $i$-th constraint of \eqref{eq:wc_cons} can be replaced with a copositive minimization problem, which yields the constraint
\begin{equation*}
\hspace{-2mm}\begin{array}{clll}
\quad\quad \;0\geq &\inf &  c_i(\bm x) + \tb ^\top \bm \psi_i + (\tb \circ \tb) ^\top \bm \phi_i + \tau_i \\
&\st&\displaystyle \tau_i \in \RR ,\; \bm \psi_i,\bm \phi_i \in\RR^{\Jh},\;\bm \gamma_i \in\RR^{LQ} \\
&&\begin{bmatrix} \Sb^\top \diag(\bm\phi_i) \Sb - \Ab_i(\bm{x}) ^\top \Ab_i(\bm x) - \diag\left([\bm{\gamma}_i^\top\;\bm 0^\top]^\top\right)& \frac{1}{2}\left(\Sb^\top \bm\psi_i - \bb_i( \bm x) +[\bm{\gamma}_i^\top\;\bm 0^\top]^\top \right)
\\
\frac{1}{2}\left(\Sb^\top \bm\psi_i- \bb_i( \bm x) +[\bm{\gamma}_i^\top\;\bm 0^\top]^\top\right)^\top & \tau_i
\end{bmatrix} \succeq_{\mathcal C} \bm 0.
\end{array}\\
\end{equation*}
The constraint is satisfied if and only if there exist decision variables $\tau_i \in \RR$, $\bm \psi_i$, $\bm \phi_i \in\RR^{\Jh}$, and  $\bm \gamma_i \in\RR^{LQ}$ such that the constraint system 
\begin{equation}
\label{eq:cons_copos_refor}
\hspace{-2mm}\begin{array}{clll}
& c_i(\bm x) + \tb ^\top \bm \psi_i + (\tb \circ \tb) ^\top \bm \phi_i + \tau_i \leq 0,\\
&\begin{bmatrix} \Sb^\top \diag(\bm\phi_i) \Sb - \Ab_i(\bm{x}) ^\top \Ab_i(\bm x) - \diag\left([\bm{\gamma}^\top\;\bm 0^\top]^\top\right)& \frac{1}{2}\left(\Sb^\top \bm\psi_i - \bb_i( \bm x) +[\bm{\gamma}_i^\top\;\bm 0^\top]^\top \right)
\\
\frac{1}{2}\left(\Sb^\top \bm\psi_i- \bb_i( \bm x) +[\bm{\gamma}_i^\top\;\bm 0^\top]^\top \right)^\top & \tau_i
\end{bmatrix} \succeq_{\mathcal C} \bm 0
\end{array}
\end{equation}
is satisfied. Therefore the $i$-th constraint of \eqref{eq:wc_cons} can be replaced by the constraint system \eqref{eq:cons_copos_refor}. The procedure for linearization of the quadratic terms $\Ab_i(\bm{x}) ^\top \Ab_i(\bm x)$ is analogous to the method presented in Theorem~\ref{thm:CP_RQCQP}.

\subsection{Non-Convex Terms in the Objective Function}
\label{extend:nonconvex_obj}

All exactness results in this paper extend immediately to the setting where the objective function in \eqref{eq:RO} involves non-convex quadratic terms in the uncertainty $\bm{\xi}$. Specifically, we consider the objective function
\begin{equation*}
Z(\bm x)=\sup_{\bm{\xi}\in\Xi}\norm{\bm A (\bm x) \bm \xi}^2 + \bm{\xi}^\top\bm D(\bm x)\bm{\xi}+ \bm b (\bm x)^\top\bm \xi + c (\bm x),
\end{equation*}
where $\bm D(\bm x):\mathcal X\rightarrow \mathbb S^{K}$ is a matrix-valued affine function of $\bm x$. We can still use Theorem \ref{thm:burer} to reformulate $Z(\bm x)$ as the optimal value of a copositive program. By following the steps of Proposition \ref{prop:CPP} and Theorem~\ref{thm:CP_RQCQP}, the copositive programming reformulation is obtained by replacing the last constraint in \eqref{eq:RO_COP} with the copositive constraint
\begin{equation*}
\begin{array}{clll}
&\begin{bmatrix} \Sb^\top \diag(\bm\phi) \Sb - \bm H -\bm{\mathcal D}(\bm x)- \diag\left([\bm{\gamma}^\top\;\bm 0^\top]^\top\right)& \frac{1}{2}\left(\Sb^\top \bm\psi - \bb( \bm x) +[\bm{\gamma}^\top\;\bm 0^\top]^\top \right)
\\
\frac{1}{2}\left(\Sb^\top \bm\psi- \bb( \bm x) +[\bm{\gamma}^\top\;\bm 0^\top]^\top\right)^\top & \tau
\end{bmatrix} \succeq_{\mathcal C} \bm 0,
\end{array}
\end{equation*}
where 
\begin{equation*}
\begin{array}{cc}
\bm{\mathcal D}(\bm x)=\begin{bmatrix}
\bm 0 & \cdots & \bm 0 &  \bm 0 & \cdots & \bm 0 & \bm 0\\
\vdots & \ddots & \vdots & \vdots & \ddots & \vdots & \vdots  \\
\bm 0 & \cdots & \bm 0 &  \bm 0 & \cdots & \bm 0 & \bm 0 \\
\bm 0 & \cdots & \bm 0 &  \bm 0 & \cdots & \bm 0 & \bm{D}(\bm x)\\
\end{bmatrix}\in\mathbb S^{\Kh}.
\end{array}
\end{equation*}
We omit the details for the sake of brevity.

\section{Numerical Experiments} 
\label{sec:experiments}
In this section, we assess the performance of the SDP approximations presented in Section~\ref{sec:SDP_formulation}. All optimization problems are solved using the YALMIP interface \cite{yalmip} on a 16-core 3.4 GHz computer with 32 GB RAM. We use MOSEK 8.1 to solve SDP formulations, and CPLEX 12.8 to solve integer programs and non-convex quadratic programs.

\subsection{Least Squares}
\label{sec:least_squares}
The classical least squares problem seeks an approximate solution $\bm x$ to an overdetermined linear system $\bm A\bm x=\bm b$ which minimizes the \emph{residual} $\|\bm A\bm x-\bm b\|^2$. This yields the following quadratic program:
\begin{equation*}
\begin{array}{clll}
\minimize& \|\bm A\bm x-\bm b\|^2\\
\subjectto & \bm x\in\RR^N.
\end{array}
\end{equation*}
The solution to this problem can be very sensitive to perturbations in the input data $\bm A\in\RR^{M\times N}$ and $\bm b\in\RR^M$~\cite{elden1980perturbation,golub2012matrix}. To address the issue of parameter uncertainty, El Ghaoui and Lebret \cite{el1997robust} recommend solving the following robust optimization problem:
\begin{equation}
\label{eq:RLS}
\begin{array}{clll}
\minimize &\displaystyle\sup_{(\bm U,\bm v)\in\mathcal U} \|(\bm A + \bm U)\bm x-(\bm b+\bm v)\|^2\\
\subjectto &\displaystyle \bm x\in\RR^N.
\end{array}
\end{equation}
Here, the goal is to find a solution $\bm x$ that minimizes the worst-case residual when the matrix $\bm U$ and the vector~$\bm v$ can vary within the prescribed uncertainty set $\mathcal U$. A tractable SDP reformulation of this problem is derived  in \cite{el1997robust} for problem instances where the uncertainty set is given by the Frobenius norm ball
\begin{equation*}
\label{eq:uncertainty_set_EG}
\mathcal B(r)=\left\{(\bm U,\bm v)\in\RR^{M\times N}\times\RR^M:\left\|\left[\bm U^\top\;\bm v\right]\right\|_F\leq r\right\}.
\end{equation*}
We consider the case when the uncertainty set is a polytope,
and compare our SDP scheme with the state-of-the-art approximate \slemma method described in Section \ref{sec:S-lemma}. We also compare our method with the approximation scheme proposed by Bertsimas and Sim~\cite{bertsimas2006tractable}, where the worst-case quadratic term in \eqref{eq:wc_functions} is replaced with an upper bounding function. Minimizing this upper-bounding function over $\bm x$ yields an approximate solution to the RQP. We note that the robust least squares problem can be solved to optimality using Benders' constraint generation method \cite{blankenship1976infinitely}. However, doing so entails solving a non-convex quadratic optimization problem at each step to generate a valid cut, which becomes intractable when $M$ and $N$ become large. 

In our experiment, we  consider the case where the uncertainty affects only the right-hand side vector $\bm b$ (\ie, $\bm U=\bm 0$). We assume that the uncertain parameter $\bm v$ depends affinely on $N_f$ \emph{factors} represented by $\bm \xi \in \RR^{N_f}$, where $N_f < M$. Specifically, we consider the uncertainty set
\begin{equation*}
\label{eq:uncertainty_LS}
\mathcal U=\left\{\bm v\in\RR^{M}:\bm v = \bm F\bm \xi,\; \bm \xi \in \RR^{N_f},\; \norm{\bm \xi}_\infty \leq 1,\; \norm{\bm \xi}_1 \leq \rho N_f\right\},
\end{equation*}
where $\bm F \in \RR^{M\times N_f}$ is the factor matrix and $\rho$ lies in the interval $[0,1]$. By substituting $\bm U = \bm 0$ and $\bm v = \bm F\bm \xi$ into \eqref{eq:RLS}, the resulting robust problem constitutes an instance of RQP \eqref{eq:RO} with the following input parameters:
\begin{equation*}
\begin{array}{c}
\bm A(\bm x)= \bm F,\quad\bm b(\bm x)= -2\bm F^\top (\bm A \bm x - \bm b) ,\quad c(\bm x) = (\bm A \bm x - \bm b)^\top (\bm A \bm x - \bm b),\\
\Xi = \left\{\bm \xi \in \RR^{N_f} : \norm{\bm \xi}_\infty \leq 1,\; \norm{\bm \xi}_1 \leq \rho N_f\right\}.
\end{array}
\end{equation*}
In order to solve the problem using our method, we modify the formulation discussed in Section \ref{sec:SDP_formulation} slightly, which leads to a tremendous reduction in the solution time. We discuss this modification in Appendix \ref{appendix:ls}.

We perform an experiment on problem instances of dimensions $M = 200$, $N = 20$ and $N_f = 30$. The experimental results are averaged over $100$ random trials generated in the following manner. In each trial, we sample the  matrix $\bm A$ and the vector $\bm b$ from the uniform distribution on $[-0.5,0.5]^{M\times N}$ and $[-0.5,0.5]^N$, respectively. Each row of the matrix $\bm F$ is sampled randomly from a standard simplex, and $\rho$ is generated uniformly at random from the interval $[0.1,0.25]$. For problems of this size, we are unable to solve the problem to optimality using Benders' method as the solver runs out of memory. Therefore, we put a time limit of $120$ seconds for each iteration of the Benders' method. By doing so, Benders' method yields a lower bound to the optimal worst-case residual, which we use as a baseline to compute the objective gaps for the approximation methods.

Table~\ref{tab:gaps_LS} summarizes the optimality gaps of the approximation methods. The results show that our method significantly outperforms the other two approximations in terms of the estimates of the worst-case residuals. While the other two approximations generate overly pessimistic estimates of the resulting worst-case residuals (with a relative difference of about 100\%), the worst-case residuals estimated using our method have negligible objective gaps.

Table~\ref{tab:times_LS} reports the solution times of finding the exact solution (using Benders' method) and the upper bounds provided by various approximation methods. It can be observed that the improvement in solution quality given by our method comes at the cost of longer solution times compared to other approximation methods. However, our method is still significantly faster than the exact Benders' method. 
We also note that while the approximation scheme described in \cite{bertsimas2006tractable} can be solved quickly, it is only valid when the uncertainty set is defined as a \emph{norm-bounded set} ($l_1 \cap l_\infty$ norm in our experiment). Our method, on the other hand, is applicable for general polyhedral uncertainty sets.
\begin{table}[h!]
	\color{black}
	\centering
	\begin{tabular}{c|rrrr}
		\multicolumn{1}{c}{} & \multicolumn{3}{c}{Objective gap}\\
		 \hline
		Statistic  & SDP & \slemma & B\&S  \\[0.0mm] \hline
		Mean & 0.0\% & 108.4\% & 99.7\%   \\
		10th Percentile & 0.0\% & 93.9\% & 80.5\% \\
		90th Percentile & 0.0\%&119.6\% & 115.3\% \\
		\hline\hline
	\end{tabular}
	\caption{\color{black} Numerical results comparing the proposed SDP approximation (`SDP'), the approximate \slemma method (`\slemma') and the approximation scheme proposed by Bertsimas and Sim~\cite{bertsimas2006tractable}~(`B\&S') for the least squares problem. The `objective gap' quantifies the increase in the worst-case residuals estimated using the approximation methods relative to the Benders' lower bound.
	\label{tab:gaps_LS}}
\end{table}
\begin{table}[h!]
	\color{black}
	\centering
	\begin{tabular}{c|rrrr}
		& Benders & SDP & \slemma & B\&S  \\[0.0mm] \hline
		Mean solution time (in secs) & 626.9 & 10.2 & 0.45\ & 0.004\\
		\hline\hline
	\end{tabular}
	\caption{\color{black} Solution times for the Benders' constraint generation method (`Benders'), the proposed SDP approximation (`SDP'), the approximate \slemma method (`\slemma') and the approximation scheme proposed by Bertsimas and Sim~\cite{bertsimas2006tractable}~(`B\&S') for the least squares problem.
		\label{tab:times_LS}}
\end{table}

\subsection{Project Management}
In this experiment, we consider the project crashing problem described in Example \ref{exa:project_crashing}, where the duration of activity~$(i,j)\in\mathcal A$ is given by the uncertain quantity $d_{ij} = (1+r_{ij})d_{ij}^0$. Here, $d_{ij}^0$ is the nominal activity duration and $r_{ij}$ represents exogenous fluctuations. We consider randomly generated project networks of size $|\mathcal V| = 30$ and \emph{order strength} 0.75,\footnote{The order strength denotes the fraction of all $|\mathcal V|(|\mathcal V|-1)/2$ possible precedences between the nodes that are enforced in the graph (either directly or through transitivity).} which gives rise to projects with an average of $67$ activities. Let $x_{ij}$ be the amount of resources that are used to expedite the activity $(i,j)$. We fix the feasible set of the resource allocation vector to 
$\mathcal X=\{\bm x\in[0,1]^{|\mathcal A|}:\mathbf e^\top\bm x\leq 
\frac{3}{4}|\mathcal A|\}$, so that at most $75\%$ of the activities can receive the maximum resource allocation. The uncertainty set of $\bm d$ is defined through a factor model as follows:
\begin{equation*}
	\mathcal D=\left\{\bm d \in \RR^{|\mathcal A|}:d_{ij} = (1+\bm f_{ij}^\top\bm{\chi})d_{ij}^0 \text{ for some }\bm{\chi}\in[0,1]^{N_f}, \quad\forall (i,j)\in\mathcal A\right\},
\end{equation*}
where  the factor size is fixed to $N_f=|\mathcal V|$. We set the nominal task durations to $\bm d^0=\mathbf e$. In each trial, we sample the  factor loading vector~$\bm f_{ij}$ from the uniform distribution on $[-\frac{1}{2N_f},\frac{1}{2N_f}]^{N_f}$, which ensures that the duration of each activity can deviate by up to $50\%$ of its nominal value. We can form the final mixed-integer uncertainty set $\Xi$ from $\mathcal D$ using the procedure described in Example \ref{exa:project_crashing} (Equation \eqref{eq:uncertainty_set_project_crashing}).

In our experiment, we compare the performance of our proposed SDP approximation with \emph{linear decision rules} (LDR) approximation scheme discussed in~\cite{chen2007robust,wiesemann2012robust} which we describe below. In Example \ref{exa:project_crashing}, for our reformulation, we model the second stage problem as the maximization problem over the binary variables~$\bm z$ (See Equation \eqref{eq:robust_project_crashing}). Alternatively, the second-stage problem can be written as the following minimization problem:
\begin{equation*}
\begin{array}{clll}
\displaystyle \minimize &\rho_{|\mathcal V|} - \rho_1\\
\subjectto & \bm \rho \in \RR^{|\mathcal V|},\\
&\rho_j - \rho_i \geq d_{ij} - x_{ij}, \forall (i,j) \in \mathcal A.
\end{array}
\end{equation*}
Here, $\bm \rho$ is second-stage variable which depends on the realization of the uncertain $\bm d$. In the LDR approximation scheme, $\bm \rho$ is restricted to be an affine function of $\bm d$, which yields a tractable conservative approximation. To assess the suboptimality of our SDP and the LDR approximation scheme, we solve the problem to optimality using Benders' constraint generation method.

Table \ref{tab:gaps_PM} presents the optimality gaps of the two approximation methods for $100$ randomly generated project networks. The solution times of all the methods are reported in Table \ref{tab:times_PM}. It can be observed that our proposed SDP approximation consistently provides near-optimal estimates of the worst-case project makespan ($\sim 2.7\%$ gaps). On the other hand, while the LDR bound can be computed quickly, the bounds are too pessimistic ($\sim 27\%$ gaps). The $10$th and $90$th percentiles of the objective gaps further indicate that the estimated makespan generated from our SDP approximation \emph{stochastically dominates} the makespan generated from the LDR approximation. In addition to a higher estimate of the worst-case makespan, the actual makespan of the resource allocation $\bm x$ generated by the LDR approximation is also higher than the ones generated by our method, as shown in the ``Suboptimality'' column in Table \ref{tab:gaps_PM}. The experimental results demonstrate that our method generates near-optimal solutions to the project crashing problem faster than solving the problem to optimality using Benders' method. 
\begin{table}[h!]
	\color{black}
	\centering
	\begin{tabular}{c|rr|rr}
		\multicolumn{1}{c}{} & \multicolumn{2}{c}{Objective gap} &
		\multicolumn{2}{c}{Suboptimality} \\ \hline
		Statistic & SDP & LDR & SDP & LDR \\[0.0mm] \hline
		Mean & 2.7\%& 26.9\% & 1.7\% & 10.0\%  \\
		10th Percentile & 2.0\% & 23.8\% & 1.3\% & 7.1\% \\
		90th Percentile & 3.2\% & 30.2\% & 2.2\% & 12.8\%   \\ 
		\hline\hline
	\end{tabular}
	\caption{\color{black} Numerical results for the proposed SDP approximation (`SDP') and the linear decision rules approximation (`LDR') for the project crashing problem. The `objective gap' quantifies the increase in the worst-case makespan estimated using the approximation methods	relative to the optimal worst-case makespan. The `suboptimality' quantifies the increase in the actual worst-case makespan of the resource allocations found using the approximation methods relative to the optimal worst-case makespan.
	\label{tab:gaps_PM}}
\end{table}
\begin{table}[h!]
	\color{black}
	\centering
	\begin{tabular}{c|rrrr}
		& Benders & SDP  & LDR  \\[0.0mm] \hline
		Mean solution time (in secs) & 518.0 & 85.0 & 0.16\\
		\hline\hline
	\end{tabular}
	\caption{\color{black} Solution times for the Benders' constraint generation method (`Benders'), the proposed SDP approximation (`SDP') and the linear decision rules approximation (`LDR') for the project crashing problem.
		\label{tab:times_PM}}
\end{table}

\subsection{Multi-Item Newsvendor}
We now demonstrate the advantage of using a mixed-integer uncertainty set over using a continuous uncertainty set in a variant of the multi-item newsvendor problem, where an inventory planner must determine the vector $\bm x\in\RR_+^N$ of order quantities for $N$ different raw-materials at the beginning of a planning period. The raw materials are used to make $K$ different types of products which are then sold to customers. The matrix $\bm F \in \RR^{N\times K}$ is such that $F_{nk}$ represents the amount of raw material $n$ required to make $1$ unit of product $k$. The demands~$\bm\xi\in\ZZ_+^K$ for these products are uncertain and are assumed to belong to a prescribed discrete uncertainty set~$\Xi$. We assume that there are no ordering costs on the raw materials but the total order quantity must not exceed a given budget $B$. Excess inventory of the $n$-th raw material incurs a per-unit holding cost of~$g_n$, while the unmet demand incurs a quadratic penalty with coefficient~$\lambda$. The quadratic penalty on the unmet demand is added to discourage stock-outs \cite{hay1970production, xia2004real}.

For any  realization of the demand vector $\bm{\xi}$, the total cost of a fixed order $\bm x$ is given by
\begin{equation*}
\begin{array}{rlll}
\displaystyle\mathcal R(\bm x,\bm{\xi})&\displaystyle=\;\sum_{n = 1}^N g_n\left(\bm x_n - \sum_{k = 1}^K F_{nk}\xi_k\right)^+ + \lambda\sum_{n=1}^N \left(\left(\sum_{k = 1}^K F_{nk}\xi_k-\bm x_n\right)^+\right)^2\\
&\displaystyle=\;\inf_{\bm y_1\in\RR^N, \bm y_2\in\RR^N}\left\{\bm g^\top\bm y_1 + \lambda\bm y_2^\top\bm y_2~:~\bm y_1\geq \bm x-\bm F \bm{\xi},\;\bm y_1 \geq \bm 0,\; \bm y_2\geq \bm F \bm{\xi}-\bm x,\; \bm y_2 \geq \bm 0\right\}.
\end{array}
\end{equation*}
Here, we use the notation $z^+$ to denote $\max\{z,0\}$. The objective of a risk-averse inventory planner is then to determine a vector of order quantities $\bm x$ that minimizes the worst-case total cost $\sup_{\bm\xi\in\Xi}\mathcal R(\bm x,\bm{\xi})$. This gives rise to the optimization problem
\begin{equation}
\label{eq:newsvendor}
\begin{array}{clll}
\minimize&\displaystyle\sup_{\bm{\xi}\in\Xi}\mathcal R(\bm x,\bm{\xi})\\
\subjectto&\displaystyle\bm x\in\RR_+^N\\
&\mathbf e^\top\bm x\leq B.
\end{array}
\end{equation}
This problem constitutes an instance of the two-stage robust quadratic optimization problem \eqref{eq:RO_two_stage}  with parameters 
\begin{equation*}
\begin{array}{lll}
&\displaystyle\bm A(\bm x)=\bm 0,\;\bm b(\bm x)=\bm 0,\;c(\bm x)=0,\;\bm P=\sqrt{\lambda}\begin{bmatrix}\bm 0&\bm 0\\ \bm 0&\mathbb I \end{bmatrix},\;\bm R=\bm 0,\;\bm r=\begin{bmatrix}\bm g\\ \bm 0\end{bmatrix},\;\vspace{.1cm}\\
&\displaystyle\bm T(\bm x)=\begin{bmatrix} -\bm F\\ \bm 0\\ \bm F \\ \bm 0 \end{bmatrix},\; \bm h(\bm x) = \begin{bmatrix}\bm x\\ \bm 0\\ -\bm x \\ \bm 0 \end{bmatrix},\; \text{and}\; \bm W = \begin{bmatrix}\mathbb I & \bm 0 \\ \mathbb I & \bm 0 \\ \bm 0 & \mathbb I \\ \bm 0 & \mathbb I
\end{bmatrix}.
\end{array}
\end{equation*}

In this experiment, we compare the performance of the SDP approximation of the optimization problem~\eqref{eq:newsvendor} when $\bm\xi$ is explicitly modeled as a discrete vector versus the model where the integer restriction on $\bm \xi$ is ignored. We consider problems with $N=8$ raw materials and $K=5$ products. We fix the vector of holding costs to $\bm g=\mathbf e$, the ordering budget to $B=20$, and the penalty constant to $\lambda = 10$. All experimental results are averaged over $100$ random trials generated in the following manner. We assume that every product uses one unit each of two randomly chosen raw materials. 
In each trial, we generate every element of $\bm G \in \RR^{2\times K}$ uniformly at random from the interval~$[0,1]$. We define the actual discrete uncertainty set ($\Xi_{\rm True}$) and the set formed by ignoring the integrality assumption ($\Xi_{\rm Cont}$)  as: 
\begin{equation*}
\Xi_{\rm True}=\left\{\bm{\xi}\in\ZZ_+^K:\bm{\xi}\leq 15\mathbf e,\;\bm G\bm{\xi}\leq 0.75 \mathbf e\right\}\quad\text{and}\quad \Xi_{\rm Cont} = \left\{\bm{\xi}\in\RR_+^K: \bm{\xi}\leq 15\mathbf e, \;\bm G\bm{\xi}\leq 0.75 \mathbf e\right\},
\end{equation*}
and solve the SDP approximations of \eqref{eq:newsvendor} with inputs $\Xi_{\rm True}$ and $\Xi_{\rm Cont}$. 
We use the Benders' constraint generation method to solve the problem to optimality.

The statistics of the optimality gaps generated by the models using $\Xi_{\rm True}$ and $\Xi_{\rm Cont}$ are reported in Table~\ref{tab:gaps_newsvendor}. The solution times of all the methods are presented in Table \ref{tab:times_newsvendor}. We observe that the model using $\Xi_{\rm True}$ as the uncertainty set provides much better estimates of the worst-case cost ($\sim 13\%$ average gap) than the model using $\Xi_{\rm Cont}$ ($\sim 85\%$ average gap). Furthermore, our proposed SDP approximation can be solved much faster than solving the problem exactly using Benders' method. For problems with integer uncertainty, these experimental results suggest that the SDP approximation which utilizes the integer restriction gives high-quality solutions in comparison to the approximation which neglects these restrictions.


\begin{table}[h!]

	\centering
	\begin{tabular}{c|rr|rr}
		\multicolumn{1}{c}{} & \multicolumn{2}{c}{Objective gap}&
		\multicolumn{2}{c}{Suboptimality} \\ \hline
		Statistic & SDP True & SDP Cont & SDP True & SDP Cont \\[0.0mm] \hline
		Mean            & 13.1\%& 85.2\% & 13.0\% & 84.9\%  \\
		10th Percentile & 0.0\% & 25.8\% & 0.0\% & 25.7\% \\
		90th Percentile & 28.4\% & 173.7\% & 27.6\% & 173.5\%   \\ 
		\hline\hline
	\end{tabular}
	\caption{\color{black} Numerical results for the SDP approximations for the newsvendor model with integer uncertainty set (`SDP True') and the model that ignores the integrality restriction (`SDP~Cont'). The `objective gap' quantifies the increase in the worst-case cost estimated using the approximation methods relative to the optimal worst-case cost. The `suboptimality' quantifies the increase in the actual worst-case cost of the order quantities found using the approximation methods relative to the optimal worst-case cost.
	\label{tab:gaps_newsvendor}}
\end{table}
\begin{table}[h!]

	\centering
	\begin{tabular}{c|rrrr}
		& Benders & SDP True & SDP Cont  \\[0.0mm] \hline
		Mean solution time (in secs) & 52.9 & 11.3 & 0.63\\
		\hline\hline
	\end{tabular}
	\caption{\color{black} Solution times for the Benders' constraint generation method for the newsvendor problem (`Benders'), the SDP approximations for the model with integer uncertainty set (`SDP True') and the model that ignores the integrality restriction (`SDP~Cont').
		\label{tab:times_newsvendor}}
\end{table}

\section{Conclusion}
\label{sec:conclusion}

The paper aims at developing a near-optimal approximation method for one- and two-stage robust quadratic programs with mixed-integer uncertain parameters. The approximation method developed in the paper is not only more general than the current state-of-the-art approximate \slemma method---since the latter only handles continuous uncertain parameters---but is guaranteed to yield a better estimate of the optimal value. Furthermore, our numerical experiments show that the difference in the performance of the two approximation method can be quite significant. Our experimental results also demonstrate the disadvantage of ignoring the integer restrictions on the uncertain parameters. In the future, it would be interesting to extend the model to the distributionally robust setting, where additional information about the distribution of the uncertain parameters is explicitly incorporated.

\newpage




\appendix
\section{Implementation of Least Squares in Section \ref{sec:least_squares}}
\label{appendix:ls}
In this section, we limit the discussion to the case when there are no discrete uncertain parameters. In the paper, we consider the uncertainty set to be of the standard form $\Xi_S:= \{ \bm \xi \geq \bm 0 : \bm S\bm \xi = \bm t \}$. However, in some cases, the uncertainty sets are more naturally represented in the inequality form $\Xi_I := \{\bm \xi: \bm S\bm \xi \leq \bm t \}$. Transforming the uncertainty set in standard form involves introducing additional variables and constraints which increases the problem size. As an example, in the least squares experiment in Section \ref{sec:least_squares}, we consider the uncertainty set $\Xi=\left\{\bm \xi \in \RR^{N_f}:\; \norm{\bm \xi}_\infty \leq 1,\; \norm{\bm \xi}_1 \leq \rho N_f\right\}$. By lifting, the uncertainty set can be equivalently written as \mbox{$\Xi_{LS}=\left\{(\bm \xi,\bm \gamma):\bm \xi \in \RR^{N_f},\;\bm \gamma \in \RR^{N_f},\; -\bm \xi \leq \bm \gamma,\; \bm \xi \leq \bm \gamma,\; \bm \gamma \leq \mathbf e,\; \mathbf e^\top \bm\gamma \leq \rho N_f\right\}$}, which is of the form $\Xi_I$. The paper \cite{burer2012copositive} presents a \emph{generalized copositive programming} (GCP) reformulation of non-convex quadratic programs over conic representable sets. In \cite{xu2018copositive}, the authors consider a conservative approximation when the cone is polyhedral, which is relevant for the polyhedral uncertainty sets that we consider. Utilizing this GCP-based approximation, the robust least squares problem 
\begin{equation*}
\begin{array}{clll}
\label{eq:RLS2}
\minimize& \displaystyle\sup_{\bm \xi \in \Xi_I} \|\bm A\bm x-(\bm b+\bm F\bm \xi)\|^2\\
\subjectto &\bm x\in\RR^N
\end{array}
\end{equation*}
that we consider in Section \ref{sec:least_squares} yields the following conservative SDP approximation:
\begin{equation}
\label{eq:ls_gcp_approx}
\begin{array}{clll}
\minimize &\tau + (\bm A \bm x - \bm b)^\top (\bm A \bm x - \bm b)\\
\subjectto&\displaystyle  \bm x\in \RR^N,\;\bm \mu\in\RR^J,\;\bm N \in\RR^{J\times J},\;\tau \in \RR\\
&\displaystyle \bm \mu \geq \bm 0,\; \bm N \geq \bm 0,\\
&\displaystyle  \begin{bmatrix}
-\bm F^\top \bm F & \bm F^\top (\bm A \bm x - \bm b)\\
(\bm A \bm x - \bm b)^\top\bm F & 0
\end{bmatrix} + \begin{bmatrix}
\bm 0 & \frac{1}{2}\bm S^\top \bm \mu \\ \frac{1}{2}\bm \mu^\top \bm S & \tau - \bm \mu^\top \bm t
\end{bmatrix}\succeq \begin{bmatrix}-\bm S^\top\\ \bm t^\top\end{bmatrix}\bm N\begin{bmatrix}-\bm S & \bm t\end{bmatrix}.\vspace{.1cm}
\end{array}
\end{equation}

We use this formulation with the uncertainty set $\Xi_{LS}$ for our experiment in Section \ref{sec:least_squares}. Skipping the conversion to the standard form generates same the objective value, but reduces the average solution time from $85$ seconds to about $10$ seconds. We emphasize that this alternate formulation might not be valid when some of the components of $\bm \xi$ are restricted to be integers. Therefore, it is not straightforward to apply it to the project management and the newsvendor experiments, both of which contain discrete uncertain parameters.


\bibliographystyle{plain}
\bibliography{bibliography} 



\end{document}